\documentclass[11pt]{amsart}

\usepackage[T1]{fontenc}
\usepackage[utf8]{inputenc}
\usepackage{lmodern}
\usepackage{stmaryrd}

\usepackage[american]{babel}
\usepackage[babel, final]{microtype}
\usepackage[all]{xy}
\usepackage{ifpdf}
\usepackage{comment}
\usepackage{multirow}

\usepackage[pdftex, dvipsnames]{xcolor}
\usepackage[pdftex, final]{graphicx}
\usepackage{pinlabel}
\usepackage[pdftex,%
    includehead,%
    includefoot,%
    nomarginpar,%
    lmargin=1in,%
    rmargin=1in,%
    tmargin=1in,%
    bmargin=1in,%
]{geometry}
\usepackage{amsmath,amsthm,amssymb,amsfonts,mathrsfs,tikz,tikz-cd,bm,verbatim,mathtools,hyperref,caption,enumitem,csquotes,float}

\hypersetup{
    colorlinks=true,
    linkcolor=NavyBlue,
    citecolor=NavyBlue
}

\usepackage[hang,flushmargin]{footmisc}

\usepackage[backend=biber,style=alphabetic,sorting=nyt, url=false,doi=false,isbn=false,backref=true,maxnames=10]{biblatex}
\DefineBibliographyStrings{english}{
  backrefpage={$\uparrow$},
  backrefpages={$\uparrow$}
}

\setlist[itemize]{noitemsep} 

\addto\extrasamerican{%
}

\newtheorem{theorem}{Theorem}[section]
\newtheorem{corollary}{Corollary}[section]
\newtheorem{proposition}{Proposition}[section]
\newtheorem{lemma}{Lemma}[section]
\theoremstyle{definition}
\newtheorem{definition}{Definition}[section]
\newtheorem{example}{Example}[section]
\newtheorem{remark}{Remark}[section]
\newtheorem{question}{Question}[section]

\makeatletter
\let\c@conjecture=\c@theorem
\let\c@corollary=\c@theorem
\let\c@proposition=\c@theorem
\let\c@lemma=\c@theorem
\let\c@definition=\c@theorem
\let\c@example=\c@theorem
\let\c@remark=\c@theorem
\let\c@equation\c@theorem
\let\c@question\c@theorem
\makeatother

\def\makeautorefname#1#2{\expandafter\def\csname#1autorefname\endcsname{#2}}
\makeautorefname{proposition}{Proposition}
\makeautorefname{lemma}{Lemma}
\makeautorefname{definition}{Definition}
\makeautorefname{example}{Example}
\makeautorefname{question}{Question}
\makeautorefname{conjecture}{Conjecture}
\makeautorefname{section}{Section}
\makeautorefname{claim}{Claim}
\makeautorefname{equation}{Equation}
\makeautorefname{remark}{Remark}
\makeautorefname{corollary}{Corollary}

\newcommand{\ZZ}{\mathbb{Z}}

\newcommand{\RR}{\mathbb{R}}

\newcommand{\RP}[1]{\mathbb{RP}^{#1}}
\newcommand{\CP}[1]{\mathbb{CP}^{#1}}

\newcommand{\cU}{\mathcal{U}}

\DeclareMathOperator{\interior}{int}

\DeclareMathOperator{\Ker}{Ker}
\DeclareMathOperator{\tb}{tb}

\title{Constructing Lagrangians from triple grid diagrams}

\author{Sarah Blackwell}
\address{Department of Mathematics, University of Virginia, Charlottesville, VA 22904}
\email{\href{mailto:blackwell@virginia.edu}{blackwell@virginia.edu}}
\urladdr{\url{https://seblackwell.com/}}

\author{David Gay}
\address{Department of Mathematics, University of Georgia, Athens, GA 30602}
\email{\href{mailto:dgay@uga.edu}{dgay@uga.edu}}

\author{Peter Lambert-Cole}
\address{Department of Mathematics, University of Georgia, Athens, GA 30602}
\email{\href{mailto:plc@uga.edu}{plc@uga.edu}}

\keywords{Lagrangian, Legendrian, grid diagram, trisection, filling, cap, symplectic, contact}
\def\subjclassname{\textup{2020} Mathematics Subject Classification}
\expandafter\let\csname subjclassname@1991\endcsname=\subjclassname
\expandafter\let\csname subjclassname@2000\endcsname=\subjclassname
\subjclass{
    57K43, 57K33 
    \hfill
    Date: \today
}

\begin{document}

\begin{abstract}
Links in $S^3$ can be encoded by grid diagrams; a grid diagram is a collection of points on a toroidal grid such that each row and column of the grid contains exactly two points.  Grid diagrams can be reinterpreted as front projections of Legendrian links in the standard contact $3$--sphere. In this paper, we define and investigate triple grid diagrams, a generalization to toroidal diagrams consisting of horizontal, vertical, and diagonal grid lines.   In certain cases, a triple grid diagram determines a closed Lagrangian surface in $\CP{2}$.  Specifically, each triple grid diagram determines three grid diagrams (row-column, column-diagonal and diagonal-row) and thus three Legendrian links, which we think of collectively as a Legendrian link in a disjoint union of three standard contact $3$--spheres. We show that a triple grid diagram naturally determines a Lagrangian cap in the complement of three Darboux balls in $\CP{2}$, whose negative boundary is precisely this Legendrian link.  When these Legendrians are maximal Legendrian unlinks, the Lagrangian cap can be filled by Lagrangian slice disks to obtain a closed Lagrangian surface in $\CP{2}$.  We construct families of examples of triple grid diagrams and discuss potential applications to obstructing Lagrangian fillings.
\end{abstract}

\maketitle

\section{Introduction and main statements} \label{sec:introduction}

This paper concerns combinatorial descriptions of Lagrangian surfaces in $\CP{2}$, analogous to grid diagrams as combinatorial descriptions of Legendrian links in $S^3$. The connection between the two settings is tied to the fact that $S^3$ has a genus one Heegaard splitting in which the $\alpha$ curve is ``vertical'' and the $\beta$ curve is ``horizontal,'' while $\CP{2}$ has a genus one trisection \cite{GK} in which the $\alpha$ curve is ``vertical,'' the $\beta$ curve is ``horizontal,'' and the $\gamma$ curve is ``diagonal'' with slope $-1$ (see \autoref{fig:tori}). Taking multiple parallel copies of these curves gives a grid on $T^2$ on which we can place marked points which encode these Legendrians or Lagrangians.

\begin{figure}[h]
    \begin{center}
    \includegraphics[width=8 cm]{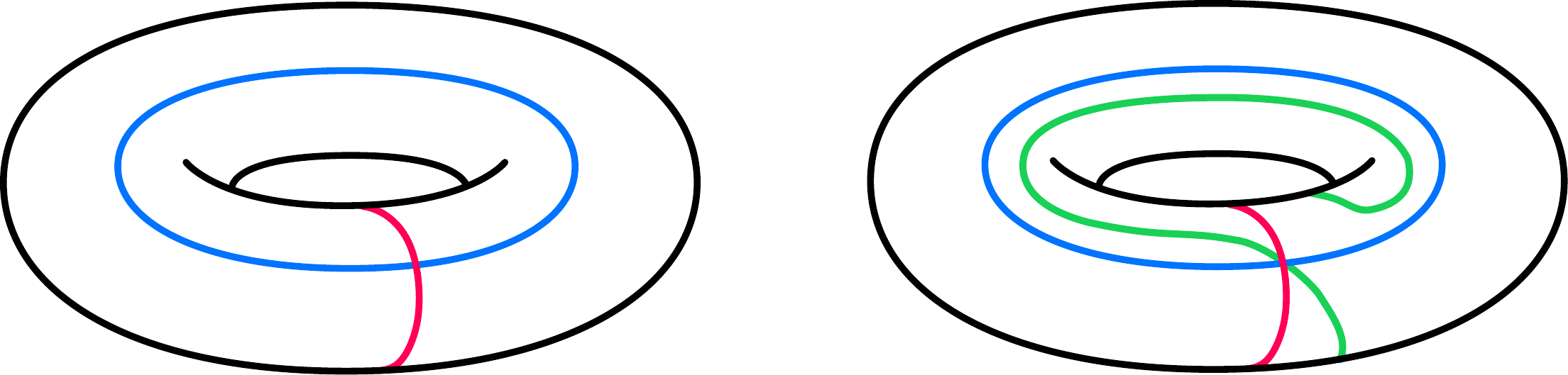}
    \end{center}
    \caption{A genus one Heegaard splitting of $S^3$ (left) and a genus one trisection of $\CP{2}$ (right).
    \label{fig:tori}}
\end{figure}

We define two slightly (but crucially) different versions of triple grid diagrams arising from this setup, which are each useful in various contexts.  Grid diagrams for links in $S^3$ are generally considered to be discrete, combinatorial objects.  It is therefore useful to define a combinatorial object mimicking the application of grid diagrams to the study of Legendrian links.  However Lagrangian surfaces are geometric objects that may live in moduli spaces and important geometric properties -- the symplectic action, monotonicity, and holomorphic disk counts -- are not invariant under small perturbations.

The following definition first appears in the first author's PhD thesis \cite{Blackwell}.
\begin{definition}
 A {\em combinatorial triple grid diagram} of grid number $n$ and size $b$ consists of:
 \begin{enumerate}
     \item a grid on the torus $T^2 = \RR^2/\ZZ^2$ consisting of three sets of lines: 
     \begin{enumerate}
         \item $n$ vertical lines $\left\{ x = \frac{k}{n} : 1 \leq k \leq n \right\}$, colored red by convention,
         \item $n$ horizontal lines $\left\{ y = \frac{k}{n} : 1 \leq k \leq n\right\}$, colored blue by convention, and
         \item $n$ diagonal lines (of slope $-1$) $\left\{x + y = \frac{k}{n} : 1 \leq k \leq n \right\}$, colored green by convention.
     \end{enumerate}
     \item $2b$ points in the complement of the $3n$ grid lines, such that in the region between any pair of adjacent lines of the same slope, there are exactly zero or two points.
 \end{enumerate}
 \end{definition}
 
See \autoref{F:BasicTGDExamples}, and further, \autoref{subsec:examples}, for several examples. Each red-blue square is divided by a green diagonal into two triangles. By convention we only place dots in the lower left triangles. We often draw these diagrams as squares where opposite edges are identified; hence the diagonals wrap around, so that there are exactly $n$ diagonals.
 
\begin{definition} A {\em geometric triple grid diagram} is a collection $\Theta$ of $2b$ points on the torus $T^2 = \RR^2/\ZZ^2$ such that for any vertical line $\{x = c\}$, any horizontal line $\{y = c\}$, or any diagonal line $\{x + y = c\}$, exactly zero or two  points of $\Theta$ lie on the line.
\end{definition}

A combinatorial triple grid diagram immediately determines a geometric triple grid diagram, by forgetting the grid lines.  Conversely, after possibly a small perturbation, it is possible to draw grid lines disjoint from a geometric triple grid diagram if the grid size $n$ is sufficiently large. For the purposes of the results that follow, either form of triple grid diagram suffices to make the statements correct, but the actual constructions start from the explicit data of the $2b$ points of a geometric triple grid diagram.

Following the work of Meier and Zupan \cite{MZ1, MZ2} on bridge trisections, and the further developments in \cite{HKM,GM}, note that a combinatorial triple grid diagram can be thought of as a special kind of multi-pointed Heegaard triple describing a bridge trisected surface, and when the points of a geometric triple grid diagram are connected horizontally, vertically and diagonally, the result can be thought of as a shadow diagram for such a surface (provided, in both cases, that all three links described by the triple grid diagram are unlinks).  
 Furthermore, symplectic surfaces in $\CP{2}$ were characterized in terms of {\em transverse bridge position} by the third author in \cite{LC-Symplectic}.

\begin{figure}
    \centering
    \includegraphics[width=10 cm]{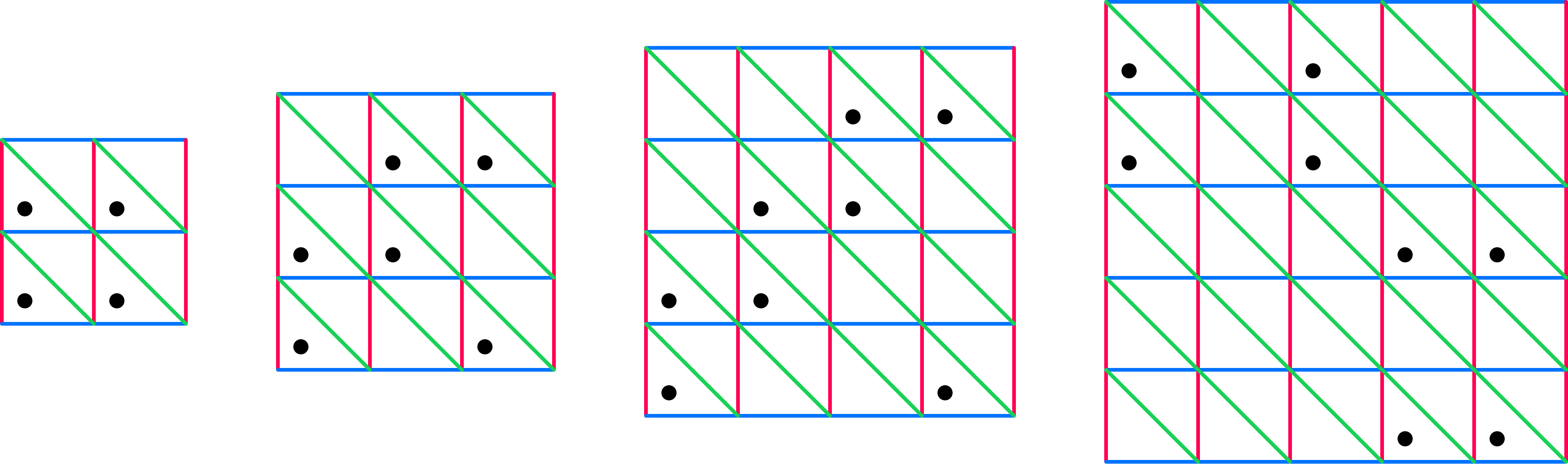}
    \caption{Some examples of combinatorial triple grid diagrams.
    \label{F:BasicTGDExamples}}
\end{figure}

We first provide executive summaries of the main results and then give more explicit details after. As we will see below, a triple grid diagram determines three Legendrian links in three copies of the standard contact $S^3$, which we identify as the boundaries of three disjoint standard balls in $\CP{2}$. 

\begin{theorem}[Executive summary]
 A triple grid diagram $D$ determines a properly embedded Lagrangian surface $L(D)$ in the complement of these three balls which is a {\em Lagrangian cap} for the disjoint union of the three associated Legendrians, in the sense of intersecting each $S^3$ in the given Legendrian and being tangent to inward pointing Liouville vector fields near each $S^3$.
\end{theorem}

\begin{corollary}[Executive summary]
 When the three Legendrian links associated to diagram $D$ are all Legendrian unlinks of Legendrian unknots of Thurston-Bennequin number $tb = -1$, then $L(D)$ can be filled with disjoint Lagrangian disks to give a closed, embedded Lagrangian surface $\overline{L}(D)$ in $\CP{2}$ determined by $D$.
\end{corollary}

Note that an outline of a proof of this corollary appears in \cite{Blackwell}, and the proof in this paper essentially follows the same ideas but has a slightly different organization.

We now give a more explicit setup that allows us to state these results precisely. By convention, the vertical lines in the grid are colored red and called $\alpha$ curves, the horizontal lines are colored blue and called $\beta$ curves, and when we draw the slope $-1$ grid lines these are colored green and called $\gamma$ curves. Note that there is an order $3$ element of $SL_2 \ZZ$, i.e. an order $3$ orientation preserving automorphism of $T^2$, which cyclically permutes $(\alpha,\beta,\gamma)$. Thus, by applying this automorphism, we can choose to view any pair of colors as horizontal and vertical, as long as the cyclic order is preserved. In other words, we can make $\beta$ vertical, $\gamma$ horizontal and $\alpha$ diagonal, or we can make $\gamma$ vertical, $\alpha$ horizontal and $\beta$ diagonal.

For each pair of colors we get a link diagram by first cyclically permuting until that pair of colors are horizontal and vertical, then connecting dots horizontally and vertically by straight line segments, and then adopting the convention that horizontal segments pass over vertical segments when they cross. Now by rotating this diagram $45^\circ$ clockwise and either smoothing corners or replacing corners with cusps so as to avoid vertical tangencies, we obtain a front diagram for a Legendrian link in $S^3$. We will call this the {\em standard Legendrianization of a grid diagram}. Thus each triple grid diagram $D$ gives three Legendrian links $\Lambda_{\alpha \beta}(D)$, $\Lambda_{\beta \gamma}(D)$ and $\Lambda_{\gamma \alpha}(D)$, which we think of as living in three different copies of the standard contact $(S^3,\xi)$, labelled $(S^3_{\alpha \beta}, \xi_{\alpha \beta})$, $(S^3_{\beta \gamma}, \xi_{\beta \gamma})$ and $(S^3_{\gamma \alpha}, \xi_{\gamma \alpha})$, respectively. This whole process is illustrated in \autoref{fig:grid_Leg}.

\begin{figure}[h]
    \centering
    \includegraphics[width=.75\linewidth]{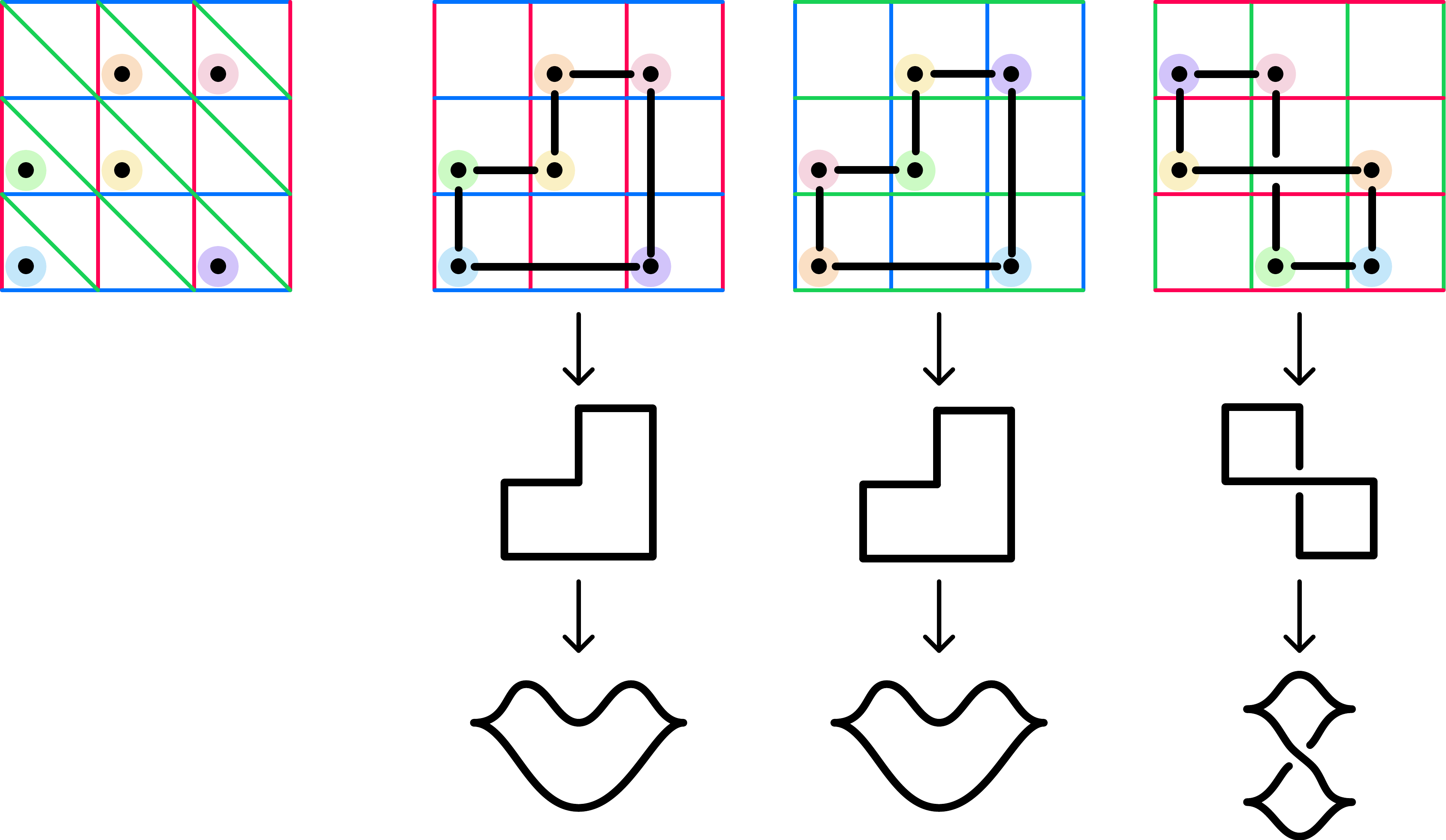}
    \caption{A (combinatorial) triple grid diagram. The right-hand side shows separately the three grids making up the triple grid diagram on the left-hand side, along with the knot represented by each grid. Colors are used to indicate where each vertex in the triple grid diagram shows up in the three individual grids. Below the grids the process for obtaining Legendrians is shown.
    \label{fig:grid_Leg}}
\end{figure}

Also naturally associated to a triple grid diagram $D$ is an abstract trivalent graph $\Gamma(D)$ with edges colored red, blue and green (or labelled $\alpha$, $\beta$ or $\gamma$) where the vertices are the dots in the diagram and two dots are connected by an appropriately colored/labelled edge when they lie in the same row, column or diagonal. See for instance \autoref{fig:graph}. The edges at each vertex are cyclically ordered by the $(\alpha,\beta,\gamma)$ order, and we use this to construct an abstract ``ribbon'' surface $R(D)$ which is a thickening of $\Gamma(D)$. Note that this is not the standard ``fat graph'' construction coming from a cyclic ordering of edges at each vertex, since this standard construction produces an orientable surface, whereas $R(D)$ may or may not be orientable. 

\begin{figure}[h]
    \centering
    \includegraphics[width=.4\linewidth]{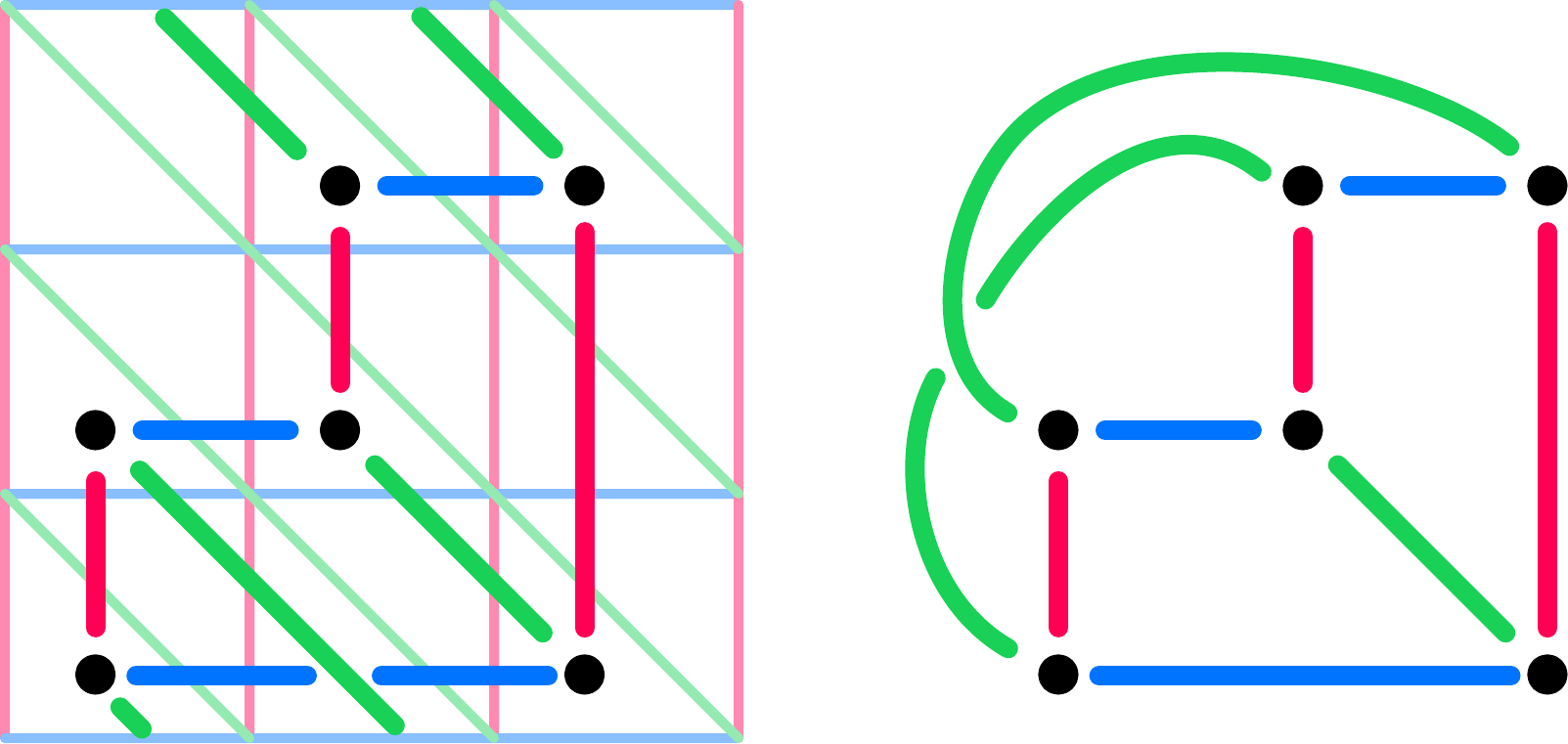}
    \caption{The trivalent graph $\Gamma(D)$ associated to the triple grid diagram $D$ shown in \autoref{fig:grid_Leg}. The left-hand side depicts the graph still on the grid, with the green edges wrapping around the torus, while the right-hand side depicts the graph abstractly.
    \label{fig:graph}}
\end{figure}

Begin with an oriented copy of $D^2_v$ for each vertex $v$ of $\Gamma(D)$ and attach an {\em orientation-reversing}, that is, {\em half-twisted}, band from $D^2_v$ to $D^2_w$ whenever there is an edge from $v$ to $w$, with the band colored/labelled the same as its corresponding edge. Attach these bands in order red, blue, green going clockwise around the boundary of each $D^2_v$. We discuss orientability of $R(D)$ in more detail in \autoref{sec:examples}.
Note that the boundary of $R(D)$ is naturally the disjoint union of three $1$--manifolds $\partial_{\alpha \beta} R(D)$, $\partial_{\beta \gamma} R(D)$ and  $\partial_{\gamma \alpha} R(D)$, based on the pairs of colors making up each component of $\partial R(D)$.

In $\CP{2}$ with projective coordinates $\{[z_1:z_2:z_2]\}$, consider the three disjoint closed balls:
\begin{align*}
 B_1 &= \{[z_1:z_2:1] \mid |z_1|^2 + |z_2|^2 \leq \tfrac{1}{5} \}, \\
 B_2 &= \{[z_1:1:z_3] \mid |z_1|^2 + |z_3|^2 \leq \tfrac{1}{5} \}, \\
 B_3 &= \{[1:z_2:z_3] \mid |z_2|^2 + |z_3|^2 \leq \tfrac{1}{5} \}.
\end{align*}
(The exact sizes are not important except that they should be small enough to accommodate an explicit construction given later, and $\tfrac{1}{5}$ is small enough.) Endow $\CP{2}$ with the standard symplectic form $\omega$ (which we express later in toric coordinates), and then let $X = \CP{2} \setminus \interior (B_1 \cup B_2 \cup B_3)$, so that $(X,\omega)$ is a symplectic $4$--manifold with three concave (hence contact) $S^3$ boundary components $(\partial_1 X, \xi_1)$, $(\partial_2 X, \xi_2)$, $(\partial_3 X, \xi_3)$. Each of these contact structures $\xi_i$ is induced by a standard radial Liouville vector field $V_i$ on $B_i$, pointing in along $\partial_i X$, and thus each $(-\partial_i X, \xi_i)$ is contactomorphic to the standard contact $(S^3,\xi)$.

Now we can restate our main theorem.

\setcounter{theorem}{2}
\begin{theorem}[Restated more precisely] \label{T:LagrangianCap}
  Given any triple grid diagram $D$ there exists a properly embedded Lagrangian surface $L(D) \subset (X,\omega)$ satisfying the following properties.
 \begin{enumerate}
  \item In a neighborhood of each $\partial_i X$, $L(D)$ is tangent  to $V_i$, so that $L(D) \cap \partial_i X$ is a Legendrian link $\Lambda_i$ in $(-\partial_i X, \xi_i)$ in $X$.
  \item There are contactorphisms taking Legendrian links to Legendrian links as follows:
  \begin{align*}
    (-\partial_1 X, \xi_1, \Lambda_1) &\to (S^3_{\alpha \beta}, \xi_{\alpha \beta}, \Lambda_{\alpha \beta}), \\
    (-\partial_2 X, \xi_2, \Lambda_2) &\to (S^3_{\beta \gamma}, \xi_{\beta \gamma}, \Lambda_{\beta \gamma}), \\
    (-\partial_3 X, \xi_3, \Lambda_3) &\to (S^3_{\gamma \alpha}, \xi_{\gamma \alpha}, \Lambda_{\gamma \alpha}).
  \end{align*}
  \item $L(D)$ is diffeomorphic to the abstract surface $R(D)$, via a diffeomorphism taking $\Lambda_1$ to $\partial_{\alpha \beta} R(D)$, $\Lambda_2$ to $\partial_{\beta \gamma} R(D)$ and $\Lambda_3$ to $\partial_{\gamma \alpha} R(D)$.
 \end{enumerate}
\end{theorem}
In short, $L(D)$ is a Lagrangian cap in $X$ for $\Lambda_{\alpha \beta} \sqcup \Lambda_{\beta \gamma} \sqcup \Lambda_{\gamma \alpha}$. We will call such a cap -- that which caps off a disjoint union of three Legendrians in three $S^3$'s -- a {\em triple cap}. 

Finally we can restate, and prove, our main corollary.

\setcounter{corollary}{3}
\begin{corollary}[Restated more precisely]  \label{C:ClosedLagrangian}
A triple grid diagram $D$ for which each of the three Legendrian links $\Lambda_{\alpha \beta}(D)$, $\Lambda_{\beta \gamma}(D)$ and $\Lambda_{\gamma \alpha}(D)$ is a Legendrian unlink of Legendrian unknots with Thurston-Bennequin number $tb = -1$ determines a closed embedded Lagrangian surface $\overline{L}(D)$ in $\CP{2}$ diffeomorphic to the result of attaching a disk to each boundary component of $R(D)$.
\end{corollary}

\begin{proof}
 The Legendrian unlink of Legendrian unknots with $tb = -1$ has a filling in $B^4$ by disjoint Lagrangian disks (\autoref{prop:unlink-filling}) and these fillings can be glued (via \autoref{prop:gluing}) to the triple cap produced in \autoref{T:LagrangianCap}.
\end{proof}

\setcounter{corollary}{4}
\begin{remark}
    More generally, as pointed out by our referee, whenever all three Legendrian links $\Lambda_{\alpha \beta}(D)$, $\Lambda_{\beta \gamma}(D)$ and $\Lambda_{\gamma \alpha}(D)$ are Lagrangian fillable we can glue the fillings to our triple cap to produce a closed Lagrangian surface.
\end{remark}

We can also obtain immersed surfaces from triple grid diagrams under more general conditions.

\setcounter{corollary}{5}
\begin{corollary}
\label{C:ImmersedCase}
Let $D$ be a triple grid diagram such that each component of each of the three Legendrian links $\Lambda_{\alpha \beta}(D)$, $\Lambda_{\beta \gamma}(D)$ and $\Lambda_{\gamma \alpha}(D)$ has rotation number $0$.  Then $D$ determines an immersed Lagrangian surface $\overline{L}(D)$ in $\CP{2}$ obtained by gluing immersed Lagrangian disks to the boundary components of $R(D)$.
\end{corollary}

\begin{proof}
The proof is identical to the proof of \autoref{C:ClosedLagrangian}, except using the immersed Lagrangian fillings from \autoref{prop:rot-0-immersed-filling}.
\end{proof}

\subsection*{Further directions}

In this paper, we focused on the triple grid diagrams compatible with the standard toric structure on $\CP{2}$.  The three grid slopes $(0,1,\infty)$ are determined by the slopes of the boundaries of the compressing disks on $T^2$.  More generally, there exists almost-toric fibrations on $\CP{2}$ (indexed by Markov triples) that yield decompositions into three rational homology 4-balls.  These almost-toric fibrations can be encoded by a triple of compressing slopes as well.  Specifically, if $(a,b,c)$ is a Markov triple and we view $\alpha,\beta,\gamma$ as vectors in $\ZZ^2$, then the three compressing slopes are determined, modulo the action of $SL(2,\ZZ)$, by the equation
\[a^2 \alpha + b^2 \beta + c^2 \gamma = 0\]
Therefore, just as grid diagrams for (Legendrian) knots in $S^3$ generalize to knots in lens spaces, there is an immediate generalization of triple grid diagrams to these decompositions arising from almost-toric fibrations.  Etnyre, Min, Piccirillo and Roy recently reinterpreted these almost-toric fibrations as small symplectic caps of triples of universally tight contact structures on lens spaces \cite{Etnyre-Min-Piccirillo-Roy}.  A triple grid diagram should naturally determine a Lagrangian cap in these small symplectic caps.

In a future paper we will discuss the uniqueness of our constructions up to Hamiltonian isotopy; this is a subtle issue because, as mentioned above, Lagrangians are geometric objects which are sensitive to small perturbations. We will also discuss the extent to which Lagrangians ``occurring in nature'' can be shown to be Hamiltonian isotopic to Lagrangians constructed from triple grid diagrams, as well as the natural question of enumerating moves on triple grid diagrams that allow us to move between different triple grid diagrams representing appropriately ``equivalent'' Lagrangians.

\subsection*{Outline}
In \autoref{sec:Lagrangian} we give some background on Lagrangian fillings and caps of Legendrians, an in particular prove the propositions needed for the proofs of \autoref{C:ClosedLagrangian} and \autoref{C:ImmersedCase}. In \autoref{sec:proofthm} we show how to construct triple caps from triple grid diagrams, proving \autoref{T:LagrangianCap}. Finally, in \autoref{sec:examples} we discuss various examples and applications.

\subsection*{Acknowledgments}
All three authors would like to thank the Max Planck Institute for Mathematics for generous hospitality in 2019-20 when much of this work was initiated, the first author for the support of a postdoctoral position in 2022-23, and the second author for support during a visit in 2023. We also would like to sincerely thank the anonymous referees for their careful reading of our paper and their insightful comments. All three authors were supported by NSF Focused Research Group grant DMS-1664567 ``FRG: Collaborative Research: Trisections -- New Directions in Low-Dimensional Topology''. The second author was supported by NSF grant DMS-2005554 ``Smooth $4$--Manifolds: $2$--, $3$--, $5$-- and $6$--Dimensional Perspectives''.

\section{Lagrangian cobordisms, fillings, and caps} \label{sec:Lagrangian}

In this section, we review the background material on Lagrangian fillings and caps and state the gluing result (\autoref{prop:gluing}).  We also describe Lagrangian disk fillings in the cases of maximal Legendrian unlinks and immersed Lagrangian links where each component has rotation number 0.  Combining these Lagrangian fillings with the Lagrangian caps constructed in the previous section and the gluing result completes the proof of \autoref{C:ClosedLagrangian}.

\subsection{Basic definitions}

To motivate the construction, we recall some terminology and facts about Lagrangian fillings of Legendrian links.

\begin{definition}
\label{def:lagrangian-concordance}
Let $(Y,\Ker(\alpha))$ be a contact $3$-manifold with contact form $\alpha$ and $(\RR \times Y, d(e^t \alpha))$ its symplectization.  Let $\Lambda_{+}$ and $\Lambda_-$ be two Legendrian links in $(Y,\Ker(\alpha))$.  A {\it Lagrangian cobordism} (with cylindrical ends) from $\Lambda_-$ to $\Lambda_+$ is an embedded Lagrangian surface $L$ satisfying:
\begin{enumerate}
    \item $L \cap [-N,N] \times Y$ is compact,
    \item $L \cap [N,\infty) \times Y \cong [N,\infty) \times \Lambda_{+}$,
    \item $L\cap (-\infty,-N] \times Y \cong (-\infty,-N] \times \Lambda_{-}$,
\end{enumerate}
for some $N$ sufficiently large.

An {\it immersed Lagrangian cobordism} is defined similarly, as an immersed Lagrangian surface but with embedded cylindrical ends.
\end{definition}

\begin{remark}
The cylindrical ends condition of the definition is equivalent to requiring that the Liouville vector field $\partial_t$ is tangent to $L$ in the half-cylinders $(-\infty,N] \times Y$ and $[N,\infty) \times Y$.
\end{remark}

A priori, the cylindrical ends condition depends on the contact form $\alpha$ and the corresponding Liouville vector field.  However, if the ends are infinite, they are cylindrical for any chosen contact form, as explained in the following lemma.  

\begin{lemma} \label{rmk:contact}
Suppose that $L$ is a Lagrangian cobordism with cylindrical ends from $\Lambda_-$ to $\Lambda_+$ in $(R \times Y, d(e^t \alpha))$.  Let $\alpha' = f \alpha$ be another contact form for $\xi = \text{ker}(\alpha)$, where $f$ a positive function. Then the image of $L$ under the symplectomorphism $\Phi: (\RR \times Y, d(e^t \alpha)) \rightarrow (\RR \times Y, d(e^t \alpha'))$ given by the map
\[\Phi(t,x) = (-\log(f(x)) + t, x)\]
is also a Lagrangian cobordism from $\Lambda_-$ to $\Lambda_+$ with cylindrical ends.
\end{lemma}

\begin{proof}
Let $M_+ = \text{max}_{Y}(-\log f)$ and $M_- = \text{min}_{Y}(-\log f)$.  Then 
\[\Phi([-N,N] \times Y) \subset [-N + M_-,N + M_+] \times Y,\]
and furthermore
\begin{align*}
\Phi^{-1}\left( [N + M_+,\infty) \times \Lambda_+\right) &\subset [N,\infty) \times \Lambda_+, \\
\Phi^{-1}\left( (-\infty,-N+ M_-] \times \Lambda_-\right) &\subset (-\infty,-N] \times \Lambda_-.
\end{align*}
This implies that
\begin{align*}
    \Phi(L) \cap [N + M_+,\infty) \times Y &= [N + M_+,\infty) \times \Lambda_+, \\
    \Phi(L) \cap (-\infty,-N + M_-] \times Y &= (-\infty,-N + M_-] \times \Lambda_-,
\end{align*}
and $\Phi(L)$ is still a Lagrangian cobordism from $\Lambda_-$ to $\Lambda_+$ with cylindrical ends.
\end{proof}

Lagrangian cobordisms with cylindrical ends can be glued together.  We state but do not prove the following result, as it is similar to \autoref{prop:gluing} below.

\begin{proposition}
If $L_a$ is a Lagrangian cobordism with cylindrical ends from $\Lambda_1$ to $\Lambda_2$ and $L_b$ is a Lagrangian cobordism with cylindrical ends from $\Lambda_2$ to $\Lambda_3$, there is a Lagrangian cobordism from $\Lambda_1$ to $\Lambda_3$ that is the smooth, topological concatenation of $L_a$ with $L_b$
\end{proposition}

\begin{definition}
\label{def:symp-filling-cap}
A {\it strong symplectic filling} of a contact structure $(Y,\xi)$ is a compact symplectic manifold $(X,\omega)$ with $\partial X = Y$, an outward-pointing Liouville vector field $\rho$ along $Y$, such that $\xi = \Ker(\omega(\rho,-)|_Y)$.  If $(X,\omega,\rho)$ is a strong filling, let $(\widehat{X}_{\rho},\widehat{\omega})$ denote the symplectic manifold obtained by adding the half-infinite collar $[0,\infty) \times Y$ to $X$ and extending $\omega$ as $d(e^t \omega(\rho,-))$ on the half-infinite cylinder.

A {\it strong symplectic cap} $(X,\omega,\rho)$ of a contact structure $(Y,\xi)$ is defined similarly to a strong filling, except the Liouville vector field is inward-pointing along $\partial X = Y$.  Let $(\widehat{X},\widehat{\omega})$ denote the symplectic manifold obtained by adding the half-infinite cylinder $(-\infty,0] \times Y$ to $X$ and extending $\omega$ as $d(e^t \omega(\rho,-))$ on the half-infinite cylinder.
\end{definition}

\begin{definition}
\label{def:lagrangian-filling-cap}
Let $(Y,\xi)$ be a contact manifold and $L$ a Legendrian link in $(Y,\xi)$.  A {\it Lagrangian filling} of $L$ in a strong symplectic filling $(X,\omega,\rho)$ of $(Y,\xi)$ is a Lagrangian submanifold $\Lambda \subset (\widehat{X}_{\rho},\widehat{\omega})$ such that
\[ \Lambda \cap [N,\infty) \times Y = [N,\infty) \times L\]
for $N$ sufficiently large.

A {\it Lagrangian cap} in a strong symplectic cap is a Lagrangain submanifold $\Lambda \subset (\widehat{X}_{\rho},\widehat{\omega})$ such that
\[\Lambda \cap (-\infty,-N] \times Y = (-\infty,-N] \times L\]
for $N$ sufficiently large.

An {\it immersed Lagrangian filling/cap} is defined similarly, as an immersed Lagrangian surface but with embedded cylindrical ends.
\end{definition}

\begin{remark}
As in \autoref{def:lagrangian-concordance}, the definition of cylindrical ends depends on the choice of contact form $\alpha$ and associated Liouville vector field.  But as in \autoref{rmk:contact}, the {\it existence} of cylindrical ends is independent of the contact form.
\end{remark}

The key gluing result is that a Lagrangian filling and a Lagrangian cap can be glued together to obtain a closed Lagrangian surface in a closed symplectic $4$-manifold.

\begin{proposition}[Gluing]
\label{prop:gluing}
Let $(Y,\xi)$ be a contact $3$-manifold and $\Lambda \subset (Y,\xi)$ a Legendrian link.  Suppose that
\begin{enumerate}
    \item $(X_f,\omega_f)$ is a strong symplectic filling of $(Y,\xi)$ and $L_f$ is a Lagrangian filling of $\Lambda$ in $(X_f,\omega_f)$, and
    \item $(X_c,\omega_c)$ is a strong symplectic cap of $(Y,\xi)$ and $L_c$ is a Lagrangian cap of $\Lambda$ in $(X_c,\omega_c)$.
\end{enumerate}
Then, after possibly rescaling $\omega_f$,
\begin{enumerate}
    \item the symplectic cap and filling can be glued to obtain a closed, symplectic $4$-manifold $(X,\omega)$, and
    \item the Lagrangian cap and filling can be glued to obtain a closed, Lagrangian surface $L \subset (X,\omega)$.
\end{enumerate}
\end{proposition}

\begin{proof}
Let $\rho_f$ be the outward-pointing Liouville vector field for $(X_f,\omega_f)$ and $\rho_c$ the inward-pointing Liouville vector field for $(X_c,\omega_c)$.  Given a contactomorphism $\Phi: \partial X_f \rightarrow \partial X_c$ sending $\Lambda$ to $\Lambda$, we can identify the induced contact forms, up to multiplying by a positive function:
\[\alpha_f = \omega_f(\rho_f,-) = g \omega_c(\rho_c,-) = g \Phi^*(\alpha_c).\]
Let $N_f,N_c$ be sufficiently large constants guaranteeing cylindrical ends for $L_f,L_c$, respectively.  Then, as in \autoref{rmk:contact} and after possibly scaling $\omega_f$ by some small constant $C$, we can extend $\Phi$ to a symplectic embedding
\[ \Phi: [N_f,2N_f] \times Y \subset (\widehat{X}_f,C\widehat{\omega}_f) \hookrightarrow (-\infty,-N_c] \times Y \subset (\widehat{X}_c,\widehat{\omega}_c).\]
This identifies the Lagrangian cylinders over $\Lambda$ as well.  The closed symplectic manifold $(X,\omega)$ can be constructed as the union of the sublevel set of $\left\{\frac{3}{2}N_f\right\} \times Y$ in $(\widehat{X}_f,C\widehat{\omega}_f)$ and the superlevel set of $\Phi\left(\left\{\frac{3}{2}N_f\right\} \times Y \right)$ in $(\widehat{X}_c,\widehat{\omega}_c)$.
\end{proof}

\subsection{Filling the unlink}
\begin{proposition}
\label{prop:unlink-filling}
Let $\mathcal{U}_n$ be an $n$-component Legendrian unlink in $(S^3,\xi_{std})$ such that every component has $tb = -1$.  Let $W = (B^4,\omega,\rho)$ be a Liouville filling of $(S^3,\xi_{std})$.  Then $\mathcal{U}_n$ has a Lagrangian filling by disjoint Lagrangian disks in $W$.
\end{proposition}

\begin{proof}
Let $\Lambda \subset (S^3,\xi_{std} = \ker \alpha)$ be a standard Legendrian unlink $\Lambda$ of $n$ Legendrian unknots with $\tb=-1$, meaning that the front projection of $\Lambda$ consists of $n$ components in disjoint disks, each with no crossings and exactly two cusps.

See \cite{BLLLMPPST} for a survey of standard constructions of Lagrangian cobordisms. In particular, Theorem~2 in \cite{BLLLMPPST}, attributed to \cite{BST} and \cite{EHK}, implies that $\Lambda$ has a filling by disjoint Lagrangian disks in the negative symplectization $((-\infty,0] \times S^3, d(e^t \alpha))$. By the classification of Legendrian unlinks \cite{Eliashberg-Fraser}, this $\Lambda$ is Legendrian isotopic to $\mathcal{U}_n$. It follows from \cite{Chantraine} that there is a Lagrangian concordance from $\mathcal{U}_n$ to $\Lambda$, which can be glued to our filling of $\Lambda$ to get a filling of $\mathcal{U}_n$ in $((-\infty,0] \times S^3, d(e^t \alpha))$.

Now note that flowing inward along the Liouville vector field $\rho$ from the boundary of $(B^4,\omega,\rho)$ gives a symplectic embedding of $((-\infty,0] \times S^3, d(e^t \alpha))$ into $(B^4,\omega,\rho)$, and thus gives a Lagrangian filling of $\mathcal{U}_n$ in $(B^4,\omega,\rho)$.
\end{proof}

\subsection{Filling general links by immersed Lagrangian surfaces}

We next generalize to the case of arbitrary Legendrian links.  Here we will explain that the $h$-principle for formal Legendrian embeddings implies that a Legendrian link admits a filling by immersed Lagrangian disks if and only if each component has vanishing rotation number.

The contact structure $(S^3,\xi_{std})$ admits a global, nonvanishing section $\tau$ since the Euler class of $\xi_{std}$ vanishes.  In addition, since $H^1(S^3;\ZZ) = 0$, this section is unique up to homotopy.  Let $\Lambda$ be an immersed Legendrian curve, which we view as a map
\[\Lambda: S^1 \rightarrow S^3.\]
Consider the pullback bundle $\Lambda^*(\xi)$ over $S^1$.  Let $v$ be a nonvanishing vector field along $S^1$ that points in the positive direction.  Since $\Lambda$ is Legendrian, the image $d\Lambda(v)$ is a nonvanishing section of $\xi$, which pulls back to give a nonvanishing section of $\Lambda^*(\xi)$, which by abuse of notation we also denote by $v$.  In addition, the global section $\tau$ pulls back to give a section of $\Lambda^*(\xi)$.  The obstruction class $d^1(v,\tau)$ to homotoping $v$ to $\tau$ through nonvanishing sections is an element of $H^1(S^1,\ZZ) \cong \ZZ$.  This integer is called the {\it rotation number} of the Legendrian immersion $\Lambda$ and is denoted $\text{rot}(\Lambda)$.

The $h$-principle for formal Legendrian immersions (see \cite{Eliashberg-Mishachev}) implies that the rotation number is a complete invariant up to Legendrian homotopy.

\begin{proposition}
\label{prop:h-principle-legendrian}
Let $\Lambda_0,\Lambda_1$ be two $n$-component Legendrian links.  There exists a family $\Lambda_t$ of immersed Legendrian links for $t \in [0,1]$ connecting $\Lambda_0$ to $\Lambda_1$ if and only if for all $i = 1,\dots,n$
\[\text{rot}(\Lambda_{0,i}) = \text{rot}(\Lambda_{1,i})\]
where $\Lambda_{k,i}$ denotes the $i^{\text{th}}$-component of the link $\Lambda_k$.
\end{proposition}

\begin{lemma}
\label{lemma:legendrian-homotopy-trace}
If there exists a Legendrian homotopy from $\Lambda$ to $\Lambda'$, then there exists an immersed Lagrangian concordance of from $\Lambda$ to $\Lambda'$.

Moreover, if the Legendrian homotopy from $\Lambda$ to $\Lambda'$ consists of $p$ positive crossing changes and $n$ negative crossing changes, the Lagrangian concordance has $p$ positive and $n$ negative transverse self-intersections.
\end{lemma}

\begin{proof}
We can decompose a Legendrian homotopy $L_t$ into a sequence of isotopies and crossing changes.  Each of these correspond to a Lagrangian cobordism with cylindrical ends, which can be concatenated to produce the immersed Lagrangian filling.  The case of a Legendrian isotopy is covered by \cite{Chantraine} and we describe the case of the crossing change by the following local model.  

Choose Darboux coordinates $(s,x,y,z)$ such that
\[\alpha = dz - y dx \qquad d(e^s \alpha) = e^s \left( ds \wedge (dz - ydx) + dx \wedge dy \right).\]
Consider the following surfaces, parameterized by $(a,b) \in \RR^2$:
\begin{align*}
    L_1 &= (a,0,b,0)  &
    L_2 &= (a,b,0,\phi_{\epsilon}(a))
\end{align*}
where $\phi_{\epsilon}$ is a smooth step function that equals $-\epsilon$ for $a \leq -\epsilon$ and equals $\epsilon$ for $a \geq \epsilon$.  These surfaces can be immediately checked to be Lagrangian with respect to $\omega = d(e^s(\alpha))$ and have cylindrical ends with respect to the Liouville vector field $\partial_s$. Furthermore, the intersection of $L_1 \cup L_2$ as $a$ varies from $-\epsilon$ to $\epsilon$ traces a Legendrian homotopy through a crossing change.  Replacing $\phi_{\epsilon}$ by $-\phi_{\epsilon}$ changes the sign of the crossing change.  
\end{proof}

\begin{proposition}
\label{prop:rot-0-immersed-filling}
Let $\Lambda$ be an $n$-component Legendrian link in $(S^3,\xi_{std})$ such that every component has $\text{rot} = 0$.  Let $W = (B^4,\omega,\rho)$ be a Liouville filling of $(S^3,\omega_{std})$.  Then $\Lambda$ has a Lagrangian filling by immersed Lagrangian disks in $W$.
\end{proposition}

\begin{proof}
By \autoref{prop:h-principle-legendrian}, the link $\Lambda$ is Legendrian homotopic to the maximal unlink $\cU_n$, where each component has $tb = -1$ and $\text{rot} = 0$.  This latter link has a Lagrangian filling by embedded disks.  Furthermore, by \autoref{lemma:legendrian-homotopy-trace}, the trace of the Legendrian homotopy from $\cU_n$ to $\Lambda$ can be realized by an immersed Lagrangian cylinder.  This can be glued to the Lagrangian filling of $\cU_n$ to produce an immersed Lagrangian filling of $\Lambda$.
\end{proof}

\section{Constructing Lagrangian caps} \label{sec:proofthm}

The goal of this section is to prove \autoref{T:LagrangianCap}, that is, to show how to construct a Lagrangian triple cap from a (geometric) triple grid diagram. In \autoref{subsec:pointstoarcs} we first construct Legendrian arcs in each of the three handlebodies which connect the points on the torus given by the triple grid diagram. Then in \autoref{subsec:solidtori} we show how to put two of these solid tori together to create Legendrian links in $S^3$, describe front projections of these links, and use the projections to show that these Legendrians are isotopic to those produced by the standard Legendrianization of a grid diagram as described in the introduction. Finally in \autoref{subsec:cap} we construct the triple cap with boundary conditions given by these Legendrian arcs.

\subsection{From points on a torus to arcs in a solid torus} \label{subsec:pointstoarcs}

First we will work in an abstract solid torus $H = D^2 \times S^1$ with boundary $\Sigma = \partial H = S^1 \times S^1$. We will use coordinates $(p,\mu,\lambda)$ on $H$ where $p = r^2 \in [0,1]$ and $(r,\mu)$ are standard polar coordinates on $D^2$, and $\lambda$ is the angular coordinate on $S^1 = \RR / 2\pi \ZZ$. Let $\mathcal{F}$ be the foliation of $H$ by meridional disks, that is, disks tangent to the integrable plane field $\ker d\lambda$, and let $\xi$ be the positive contact structure $\ker (d\lambda + p d\mu)$ on $H$. 

Our goal, given points on $\Sigma$ coming from a triple grid diagram, is to construct two sets of arcs in $H$ connecting the points: \textit{flat arcs} and \textit{Legendrian arcs}. The flat arcs will be tangent to $\mathcal{F}$, and from these we will construct Legendrian arcs tangent to $\xi$. Note that a Legendrian arc in $(H,\xi)$ is completely determined by its \textit{front projection} onto $\Sigma$. Furthermore, the $\lambda/\mu$ slope of this front projection is constrained to lie between $0$ and $-1$, and the $p$ coordinate is recovered from the front projection by the equation $p = -d\lambda/d\mu$ (in other words, $p$ is the negative of the $\lambda/\mu$ slope).

We start with a collection of points $P_1,P'_1, \ldots P_b, P'_b$ on $\Sigma$ such that each pair $P_i$ and $P'_i$ lie on the boundary of the same meridional disk, i.e. have the same $\lambda$ coordinate. In $(p,\mu,\lambda)$ coordinates, these points have coordinates of the form $P_i = (1,m_i,l_i)$ and $P_i' = (1,m_i',l_i)$. For each such pair, let $D_i$ be the meridional disk (leaf of $\mathcal{F}$) $\{\lambda = l_i\}$, and choose an arc $A_i$ in $H$ satisfying the following four properties.
\begin{enumerate}
    \item $A_i$ is properly embedded in $D_i$ with endpoints $P_i$ and $P'_i$.
    \item $A_i$ does not pass through the center of the disk $D_i$.
    \item $A_i$ cuts $D_i$ into two components, and the area of the component not containing the center $\{0\}$ of the disk (measured by the area form $dp \wedge d\mu = 2 r dr d\theta$) is equal to the distance on the boundary of the disk between $P_i$ and $P_i'$ (measured by $d\mu$, and where we measure the distance using an arc in $\partial D_i$ which is homotopic to $A_i$ in $D_i \setminus \{0\}$).
    \item On $\{p \in [1/2,1]\}$, the arc $A_i$ is radial, i.e. has constant angular coordinate $\mu$.
\end{enumerate}
Two such arcs and the $(p,\mu,\lambda)$ coordinate system on a solid torus are illustrated in \autoref{F:ArcsInTorus}.

\begin{figure}
    \labellist
    \small\hair 2pt
    \pinlabel {$p$} at 158 3
    \pinlabel {$\mu$} at 164 36
    \pinlabel {$\lambda$} at 82 222
    \endlabellist
    \centering
    \includegraphics[width=4 cm]{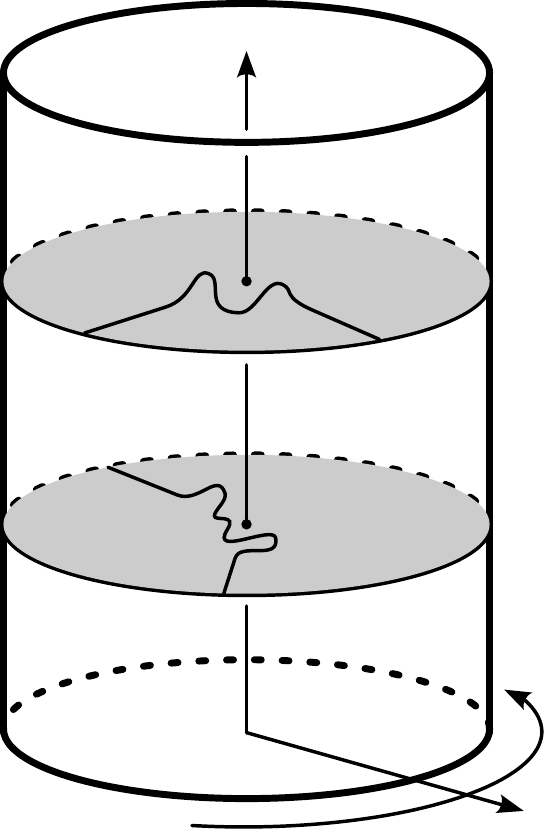}
    \caption{Arcs on meridional disks in a solid torus $H$, with radial coordinate $p$, meridional coordinate $\mu$ and longitudinal coordinate $\lambda$. The torus is illustrated as a cylinder, with the understanding that the top is glued to the bottom by the identity.} 
    \label{F:ArcsInTorus}
\end{figure}

These conditions mean that we can parameterize $A_i$ as $(p(t),\mu(t), \lambda(t))$, with $t \in [-1,1]$, satisfying the following conditions:
\begin{itemize}
    \item $\lambda(t) = l_i$,
    \item $p(t) \in (0,1]$,
    \item for $t \in [-1,-1/2]$, $p(t) = -t$ and $\mu(t) = m_i$,
    \item for $t \in [1/2,1]$, $p(t) = t$ and $\mu(t) = m'_i$,
    \item and $ \displaystyle \int_{-1}^1  p(t) \mu'(t) dt = 0$.
\end{itemize}
Now extend the domain of this parameterized curve to $t \in \RR$ as follows:
\begin{itemize}
    \item $\lambda(t) = l_i$ for all $t \in \RR$,
    \item for $t \in [-\infty,-1/2]$, $p(t) = -t$ and $\mu(t) = m_i$,
    \item and for $t \in [1/2,\infty]$, $p(t) = t$ and $\mu(t) = m'_i$.
\end{itemize}
This is a parameterized curve in $\RR^2 \times S^1$ which restricts to $D^2 \times S^1$ as the original arc in $H$. Then, define a new function 
\[\tilde{\lambda}(t) = l_i - \int_{-1}^t p(\tau) \mu'(\tau) d \tau \]
and consider the path $\tilde{A}_i$ parameterized by $(p(t), \mu(t), \tilde{\lambda}(t))$. This path enjoys two key properties. First, $\tilde{\lambda}'(t) = - p(t) \mu'(t)$, so that $\tilde{A}_i$ is Legendrian with respect to the contact structure $\xi = \ker(d \lambda + p d \mu)$. Second, $\tilde{\lambda}(t) = l_i$ for $t \in (-\infty,-1/2] \cup [1/2,\infty)$, so that $\tilde{A}_i$ agrees with $A_i$ for $p \geq 1/2$. This fact is straightforward for $t \in (-\infty,-1/2]$ but depends on the enclosed area condition to get the result for $t \in [1/2,\infty)$.

Note that the flat arcs $A_i$ can be recovered from the Legendrian arcs $\tilde{A}_i$ simply by projecting onto the appropriate meridional disks, and that in fact this projection is the standard Lagrangian projection of a Legendrian. Although the above discussion describes the Legendrian arcs as constructed from the flat arcs, in fact in the proof of the following proposition we will construct the Legendrian arcs first and then get the flat arcs as their Lagrangian projections. Some care is needed in constructing the Legendrian arcs to make sure that their Lagrangian projections are in fact embedded.

\begin{proposition} \label{P:FlatAndLegendrianTangles}
Given the initial data of $b$ pairs of points $P_i = (1,m_i,l_i)$ and $P_i' = (1,m_i',l_i)$ on $\Sigma$, with $m_i < m'_i$, and given some $\epsilon>0$ such that the $b$ intervals $[l_i-\epsilon,l_i+\epsilon]$ are disjoint, there exists a system of flat arcs $A_i$ and corresponding Legendrian arcs $\tilde{A}_i$ related as above with the following properties.
\begin{enumerate}
    \item Each $\tilde{A}_i$ lies in the $3$--dimensional wedge $\{(p,\mu,\lambda) \mid m_i-\epsilon < \mu < m_i + \epsilon, l_i - \epsilon \leq \lambda \leq l_i+\epsilon \}$.
    \item The front projection of each $\tilde{A}_i$ has the qualitative features illustrated in \autoref{F:SkewedN}, i.e. starts at $P_i$, heads up and to the left, encounters a cusp, then turns down and to the right, passes $P'_i$ on the left, encounters another cusp, and then turns up and to the left to end up at $P'_i$. 
\end{enumerate}
Furthermore, any two such choices of systems of flat arcs $A_i$ (flat $b$--component tangles) yield systems of Legendrian arcs $\tilde{A}_i$ (Legendrian $b$--component tangles) which are Legendrian isotopic rel. $\{1/2 \leq p \leq 1\}$.
\end{proposition}

\begin{figure}
    \labellist
    \small\hair 2pt
    \pinlabel {$\mu$} [t] at 100 5
    \pinlabel {$\lambda$} [r] at 5 50
    \pinlabel {$l_i - \epsilon$} [r] at 15 25
    \pinlabel {$l_i + \epsilon$} [r] at 15 78
    \pinlabel {$l_i$} [l] at 185 50
    \pinlabel {$m_i-\epsilon$} [t] at 15 15
    \pinlabel {$m'_i+\epsilon$} [t] at 185 15
    \pinlabel {$m_i$} [t] at 50 13
    \pinlabel {$m'_i$} [t] at 150 15
        \endlabellist
    \centering
    \includegraphics[width=5 cm]{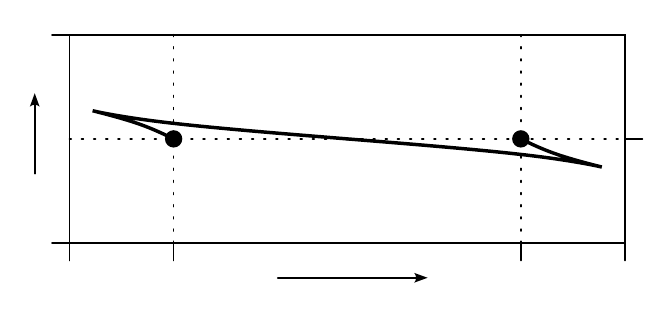}
    \caption{Front projection of a Legendrian arc in a solid torus, with all slopes between $0$ and $-1/2$.}
    \label{F:SkewedN}
\end{figure}

\begin{proof}
    The fact that any two Legendrian arcs with front projections as in \autoref{F:SkewedN} are Legendrian isotopic simply follows from the fact that the qualitative features described determine the planar isotopy type of the front diagram, and we can easily arrange to keep the slopes between $0$ and $-1/2$ so that the arcs are radial  for $p \in [1/2,1]$. Another feature of the front diagram is that the slope can be arranged to be monotonically increasing for the first half of the path from $P_i$ to $P'_i$, and then monotonically decreasing for the second half of the path. This means that the Lagrangian projection to the $(p,\mu)$ coordinate plane is embedded, and it is precisely this Lagrangian projection, when placed in the meridional disk at $\lambda = l_i$, that is the flat arc $A_i$. If we are not careful about the slopes we may not get embedded flat arcs.
\end{proof}

\begin{definition}
    A system of flat arcs $\{A_i\}$ and associated system of Legendrian lifts $\tilde{A}_i$ as in \autoref{P:FlatAndLegendrianTangles} will be called a {\em trivial flat tangle} with {\em associated trivial Legendrian tangle}.
\end{definition}

\subsection{Putting two solid tori together} \label{subsec:solidtori}

We now consider $S^3$ built by gluing two solid tori $H_\alpha$ and $H_\beta$ together, each a copy of the standard solid torus $H$ discussed in the preceding section, now with coordinates $(p_\alpha,\mu_\alpha,\lambda_\alpha)$ and $(p_\beta,\mu_\beta,\lambda_\beta)$.  In the standard Heegaard splitting, we see $S^3$ as $H_\alpha \cup H_\beta$ where the coordinates on $\Sigma_\alpha = \partial H_\alpha$ and $\Sigma_\beta = \partial H_\beta$ are related by the (orientation reversing) identifications:
\begin{align*}
    \mu_\beta &= \lambda_\alpha, \\
    \lambda_\beta &=  \mu_\alpha.
\end{align*}
The two contact structures $\xi_\alpha$ and $\xi_\beta$ coming from the contact structure $\xi$ on $H$ then glue together to give the standard contact structure on $S^3$.

Now, given a collection of $2b$ points on the torus $\Sigma_\alpha = \Sigma_\beta$ forming, in the $\mu_\alpha$ direction, pairs with the same $\lambda_\alpha$ coordinates and, in the $\mu_\beta= \lambda_\alpha$ direction, pairs with the same $\lambda_\beta = \mu_\alpha$ coordinates, we can construct trivial flat tangles and associated trivial Legendrian tangles  in each solid torus $H_\alpha$ and $H_\beta$. Since the Legendrian tangles are radial near the boundary of each solid torus, they glue together to form a closed Legendrian link in $(S^3_{\alpha \beta}, \xi_{\alpha \beta})$.

\begin{proposition} \label{P:JoiningLegendrianTangles}
    Given a geometric triple grid diagram $D$, the Legendrian link constructed by gluing together the trivial Legendrian tangles coming from \autoref{P:FlatAndLegendrianTangles} in the solid tori $H_\alpha$ and $H_\beta$ is Legendrian isotopic to the standard $\alpha \beta$ Legendrianization $\Lambda_{\alpha \beta}$ of the grid diagram. Cyclically permuting the three colors gives the same result for $\beta \gamma$ and for $\gamma \alpha$.
\end{proposition}

\begin{proof}
    The Legendrian produced by gluing together our Legendrian tangles has a front projection obtained from the grid diagram by replacing each vertical and horizontal arc by a vertical or horizontal copy of the front diagram shown in \autoref{F:SkewedN}. \autoref{F:TwoFrontsForTrefoil} compares the resulting front diagram for a simple grid diagram of the trefoil to the front diagram coming from the standard Legendrianization. Note that to implant the local front diagram in \autoref{F:SkewedN} into \autoref{F:TwoFrontsForTrefoil}, we need to remember that, when the boundary of the $H_\alpha$ solid torus is identified with the torus on which the $\alpha \beta$ grid diagram is drawn, the $\mu$ axis points up and the $\lambda$ axis points left, while with respect to the $H_\beta$ solid torus, the $\mu$ axis points to the left and the $\lambda$ axis points up. In other words, red is a copy of \autoref{F:SkewedN} rotated counterclockwise $90^\circ$ while blue is a copy of \autoref{F:SkewedN} reflected across the vertical axis. The two fronts illustrated are related by Legendrian Reidemeister I and Reidemeister II moves, and this example is sufficient to illustrate the general case.
\end{proof}
\begin{figure}
    \centering
    \includegraphics[width=10 cm]{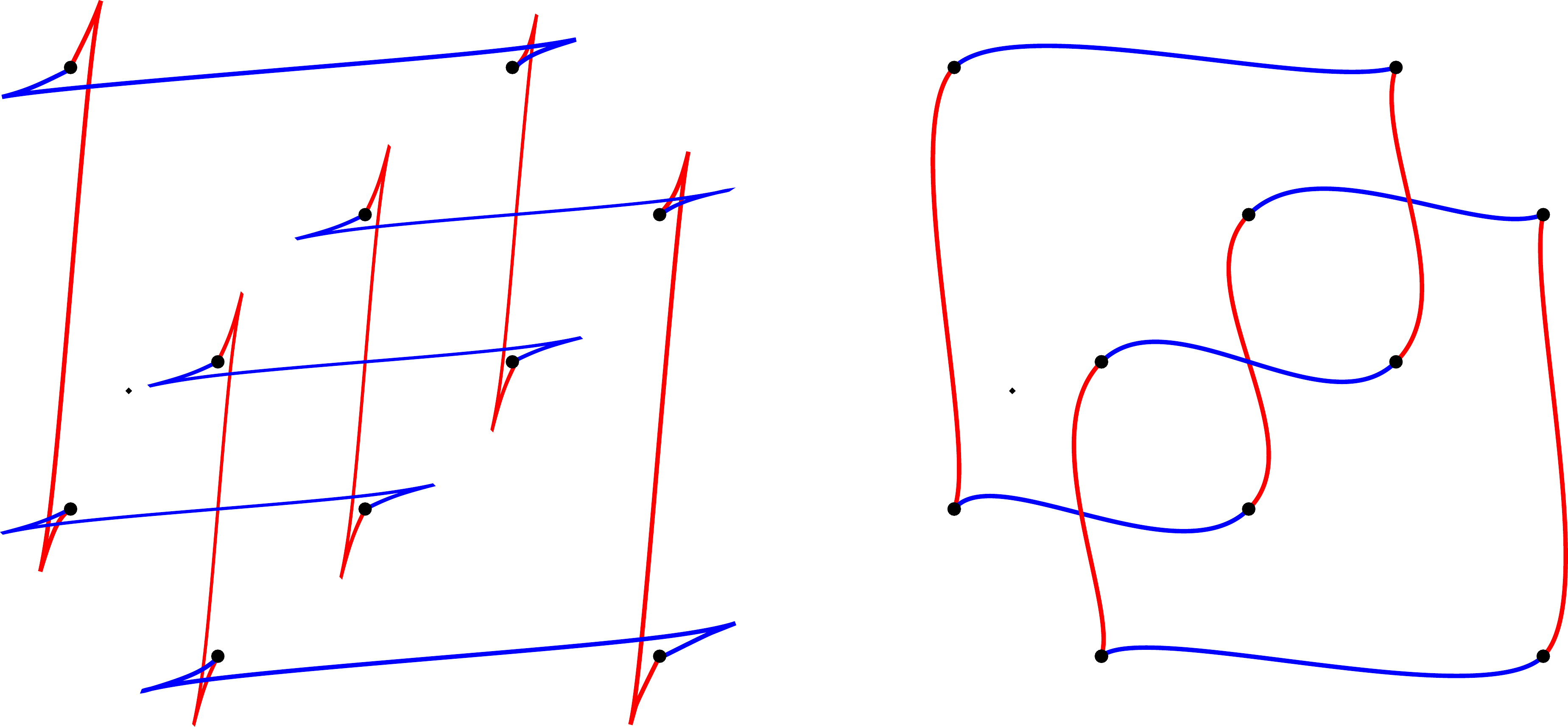}
    \caption{Two different front projections for a trefoil, the one on the left coming from a grid diagram via our Legendrian tangle construction and the one on the right coming from the same grid diagram via the standard Legendrianization process.}
    \label{F:TwoFrontsForTrefoil}
\end{figure}

\subsection{Constructing the triple cap} \label{subsec:cap}

We first establish standard toric coordinates on $\CP{2}$ so that we can work with $\CP{2}$ via its moment map image. Starting with homogeneous coordinates $[z_1:z_2:z_3]$ on $\CP{2}$, consider standard (scaled) toric coordinates:
\begin{align*}
 p_1 &= \frac{3 |z_1|^2}{|z_1|^2 + |z_2|^2 + |z_3|^2} \in [0,3], \\
 q_1 &= \arg(z_1/z_3) \in \RR/2 \pi \ZZ, \\
 p_2 &= \frac{3 |z_2|^2}{|z_1|^2 + |z_2|^2 + |z_3|^2} \in [0,3], \\
 q_2 &= \arg(z_2/z_3) \in \RR/2 \pi \ZZ.
\end{align*}
The standard moment map is the map $\Pi: \CP{2} \to \Delta$ given by $\Pi([z_1:z_2:z_3]) = (p_1,p_2)$, where $\Delta$ is the right triangle in $\RR^2$ with vertices at $(0,0)$, $(0,3)$ and $(3,0)$. This is illustrated in \autoref{F:SimpleMomentMapImage}. With respect to these coordinates, the standard symplectic structure on $\CP{2}$ (up to scale) is $\omega = dp_1 \wedge dq_1 + dp_2 \wedge dq_2$. 
\begin{figure}
    \labellist
    \small\hair 2pt
    \pinlabel {$p_1$} at 305 8
    \pinlabel {$p_2$} at -15 280
    \endlabellist
    \centering
    \includegraphics[width=3 cm]{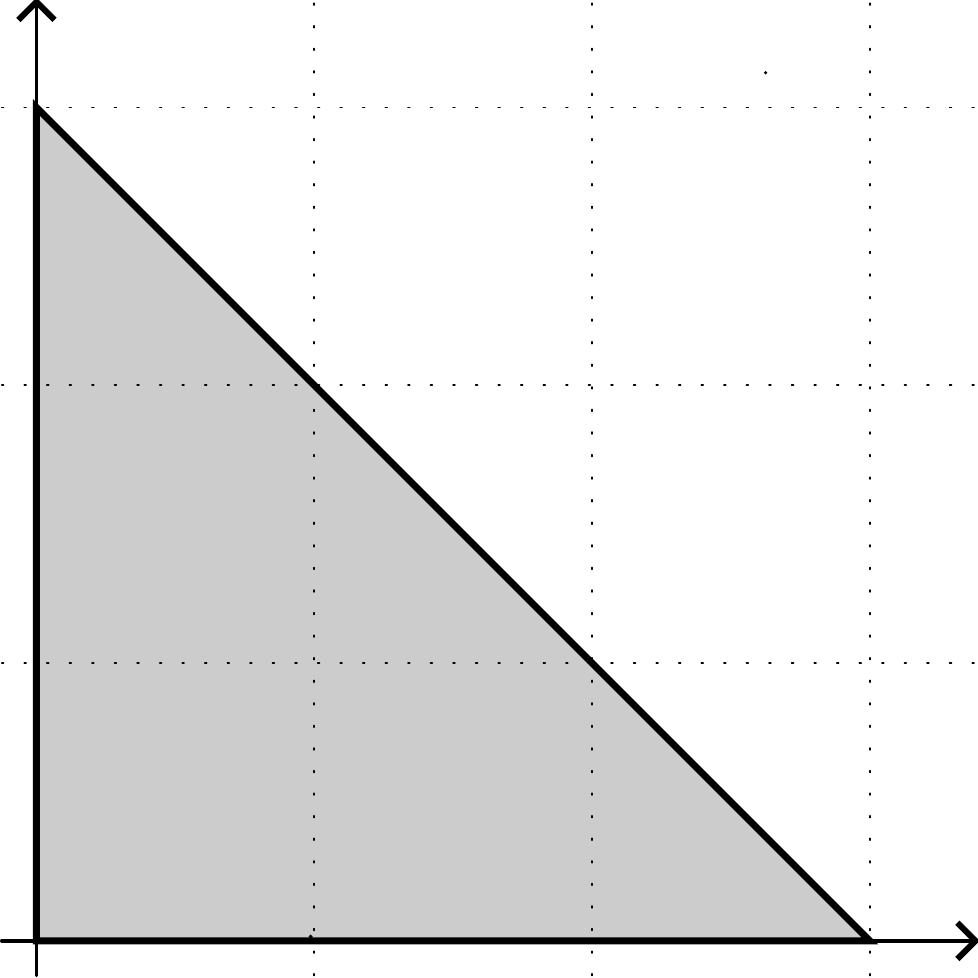}
    \caption{Moment map image for $\CP{2}$.}
    \label{F:SimpleMomentMapImage}
\end{figure}

We henceforth view our triple grid diagram $D$ as living on the torus $\Sigma = \Pi^{-1}(1,1)$ with coordinates $(q_2,q_1)$, where this orientation is chosen so that the product coordinates $(p_1,p_2) \times (q_2,q_1)$ on $\Delta \times \Sigma$ give the same orientation as our correctly oriented toric coordinates $(p_1,q_1,p_2,q_2)$ on $\CP{2}$. However, we draw the $q_2$ axis vertical and oriented upwards and the $q_1$ axis horizontal and oriented to the left. (We know this is confusing but ask the reader to bear with us, it is worth the time to get orientation conventions correct.) The horizontal lines in our grid are thus parallel to the $q_1$--axis, or given by the level sets of the function $q_2$, the vertical lines are parallel to the $q_2$--axis, or given by the level sets of the function $q_1$, and the diagonal lines, which we have been calling the ``slope $-1$ diagonals'', are given by the level sets of $q_1-q_2$. This is illustrated in \autoref{F:qCoordinates}.
\begin{figure}
    \labellist
    \small\hair 2pt
    \pinlabel {$q_1$} at -15 100
    \pinlabel {$q_2$} at 80 200
    \endlabellist
    \centering
    \includegraphics[width=4 cm]{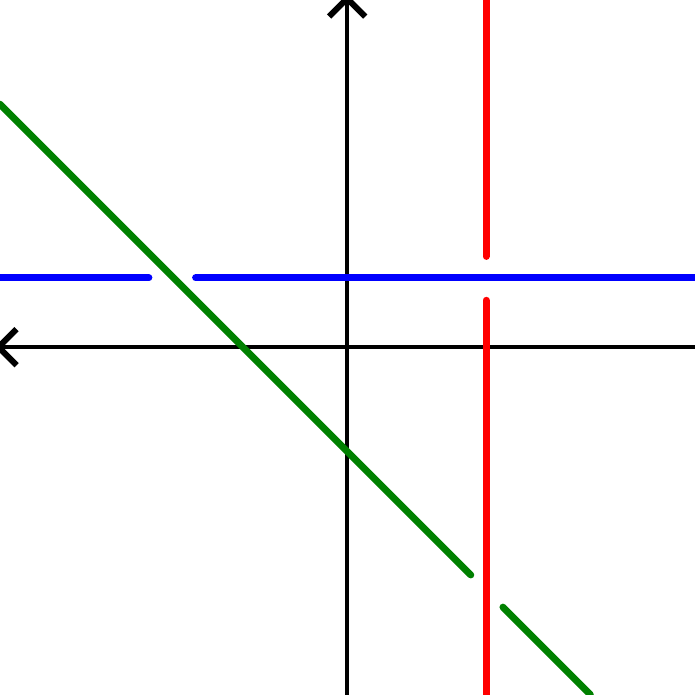}
    \caption{The $(q_2,q_1)$--plane.}
    \label{F:qCoordinates}
\end{figure}

The standard trisection of $\CP{2} = X = X_1 \cup X_2 \cup X_3$ is explicitly described as:
\begin{align*}
 X_1 &= \{(p_1,q_1,p_2,q_2) \mid p_1 + 2 p_2 \leq 3, 2 p_1 + p_2 \leq 3  \}, \\ 
 X_2 &= \{(p_1,q_1,p_2,q_2) \mid p_1 + 2 p_2 \geq 3, p_1 \leq p_2 \}, \\
 X_3 &= \{(p_1,q_1,p_2,q_2) \mid 2 p_1 + p_2 \geq 3, p_1 \geq p_2 \}. 
\end{align*}
The pairwise intersections are the solid handlebodies (solid tori) $H_{ij} = X_i \cap X_j$; a more conventional trisector's labelling is $H_\alpha = H_{31}$, $H_\beta = H_{12}$, $H_\gamma = H_{23}$. The boundary of each of these solid tori is the central genus one surface $\Sigma$. The core circles of the solid tori are given in coordinates as follows:
\begin{align*}
    C_\alpha &= \{(p_1,q_1,p_2,q_2) \mid (p_1,p_2) = (3/2,0)\} \subset H_\alpha, \\
    C_\beta &= \{(p_1,q_1,p_2,q_3) \mid (p_1,p_2) = (0,3/2)\} \subset H_\beta, \\
    C_\gamma &= \{(p_1,q_1,p_2,q_2) \mid (p_1,p_2) = (3/2,3/2)\} \subset H_\gamma.
\end{align*}
(Since the $(p_1,p_2)$ points $(3/2,0)$, $(0,3/2)$ and $(3/2,3/2)$ are in the interiors of edges of $\Delta$, the $(q_1,q_2)$ torus is collapsed to a circle in each case.) This is all illustrated in \autoref{F:MomentMapTrisection}.
\begin{figure}
    \labellist
    \small\hair 2pt
    \pinlabel {$H_\beta = H_{12}$} [r] at 0 100
    \pinlabel {$H_\alpha = H_{31}$} [t] at 100 5
    \pinlabel {$H_\gamma = H_{23}$} [l] at 198 130
    \pinlabel {$\Sigma$} [l] at 219 110
    \pinlabel {$C_\alpha$} [t] at 153 28
    \pinlabel {$C_\beta$} [r] at 30 148
    \pinlabel {$C_\gamma$} [bl] at 148 148 
    \pinlabel {$X_1$} at 80 80
    \pinlabel {$X_2$} at 80 170
    \pinlabel {$X_3$} at 170 80
    \endlabellist
    \centering
    \includegraphics[width=6 cm]{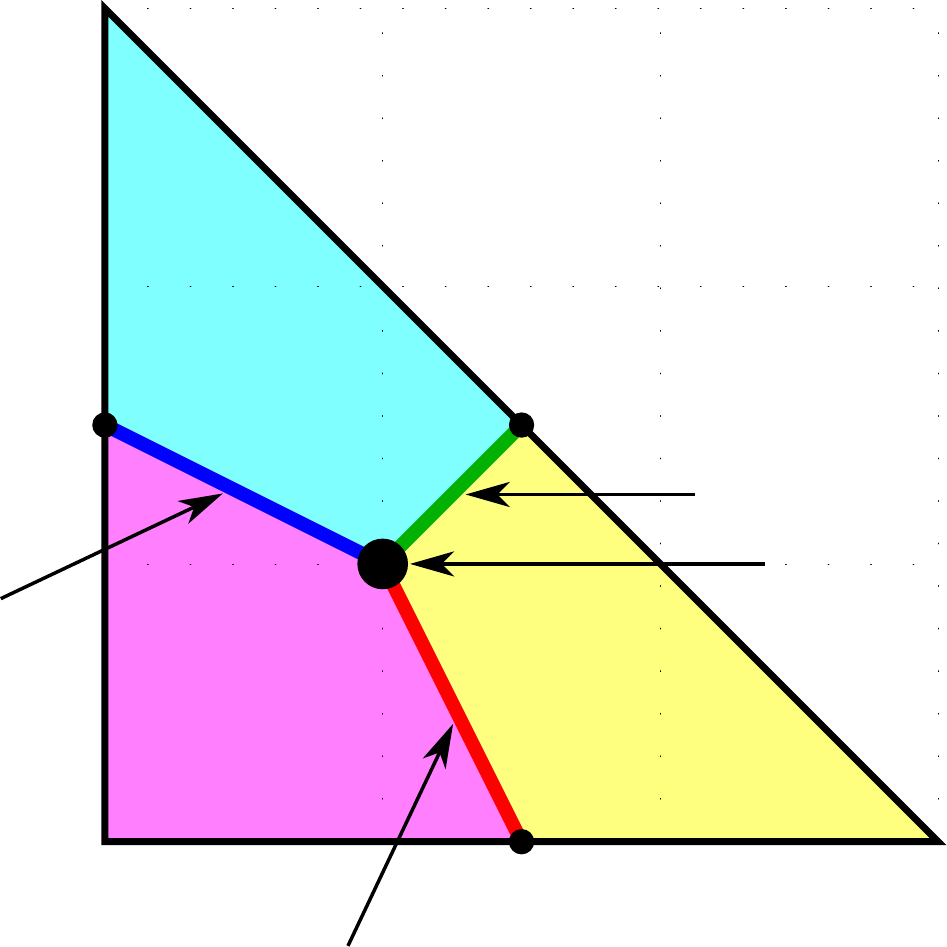}
    \caption{The standard trisection seen via the moment map.} 
    \label{F:MomentMapTrisection}
\end{figure}

Although our overall construction is naturally associated with a trisection of $\CP{2}$, in which $\Sigma$ is the central surface, we will in fact see that to give the construction as explicitly as possible it will be useful to decompose $\CP{2}$ into many more than three pieces. It is important to note, however, that our entire construction will be invariant under the order $3$ cyclic permutation $[z_1:z_2:z_3] \mapsto [z_2:z_3:z_1] \mapsto [z_3:z_1:z_2] \mapsto [z_1:z_2:z_3]$, which also cyclically permutes the vertices of $\Delta$ and induces the order $3$ cyclic permutation of our grid diagram. We will exploit this so that in some sense we only give one third of the construction and then apply this cyclic permutation.

First we note that the three $4$--balls defined in the introduction are given in toric coordinates as:
\begin{align*}
 B_1 &= \{p_1 + p_2 \leq 1/2\} \\
 B_2 &= \{3 - p_2 \leq 1/2\} \\
 B_2 &= \{3 - p_1 \leq 1/2\}
\end{align*}
These are illustrated in \autoref{F:3BallsInDelta}, which thus also illustrates our $4$--manifold with three boundary components, $X$, via its moment map image, which is a hexagon obtained by cutting three small corners off of the right triangle $\Delta$. 
\begin{figure}
    \centering
    \includegraphics[width=3 cm]{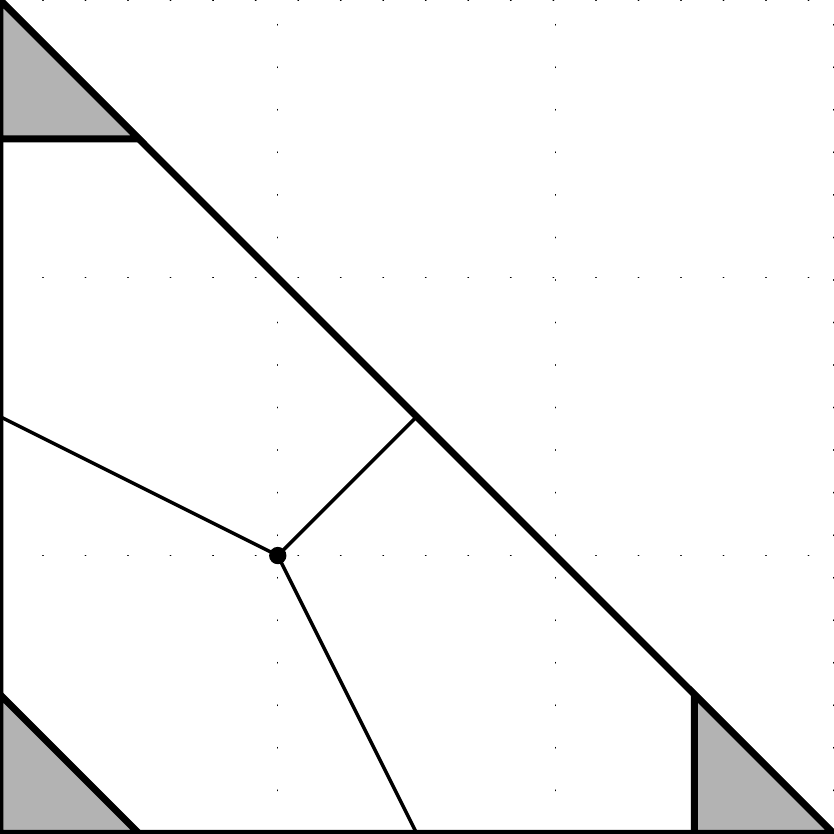}
    \caption{The images of three balls under the moment map, along with the trisection for reference.} 
    \label{F:3BallsInDelta}
\end{figure}
The three inward-pointing Liouville vector fields appear in the moment map image as the three radial vector fields emanating from the corners of $\Delta$, and are given in coordinates as:
\begin{align*}
 V_1 &= p_1 \partial_{p_1} + p_2 \partial_{p_2}, \\
 V_2 &= p_1 \partial_{p_1} + (p_2 - 3) \partial_{p_2}, \\
 V_3 &= (p_1 - 3) \partial_{p_1} + p_2 \partial_{p_2}.
\end{align*}
We will also want to work with the {\em symplectic vector fields} $S_{ij} = (1/3) (V_i - V_j)$. (The factor of $1/3$ is just for convenience in coordinate expressions.) In certain regions the Lagrangians we construct will be tangent to, and invariant under flow along, one or more of these symplectic vector fields. For future reference we give them in coordinates here:
\begin{align*}
    S_{13} &= \partial_{p_1} = -S_{31}, \\
    S_{12} &= \partial_{p_2} = -S_{21}, \\
    S_{23} &= \partial_{p_1} - \partial_{p_2} = - S_{32}.
\end{align*}
These vector fields are illustrated in \autoref{F:LiouvilleAndSymplectic}.

\begin{figure}
    \labellist
    \small\hair 2pt
    \pinlabel {$V_1$} [l] at 300 212
    \pinlabel {$V_2$} [l] at 300 182
    \pinlabel {$V_3$} [l] at 300 152
    \pinlabel {$\pm S_{31}$} [l] at 300 122
    \pinlabel {$\pm S_{12}$} [l] at 300 92
    \pinlabel {$\pm S_{23}$} [l] at 300 62
    \endlabellist
    \centering
    \includegraphics[width=6 cm]{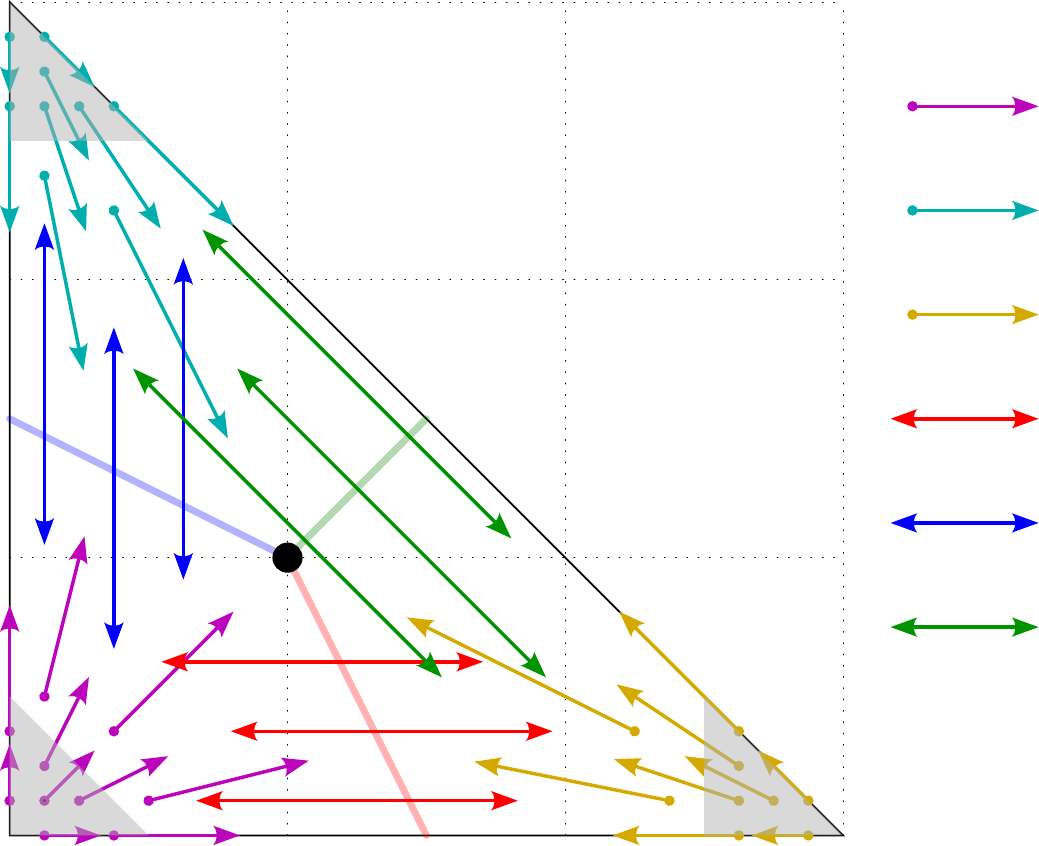}
    \caption{Three Liouville and three symplectic vector fields.} 
    \label{F:LiouvilleAndSymplectic}
\end{figure}

We have already established an orientation on $\Sigma$, namely by the ordered coordinates $(q_2,q_1)$. We will orient each solid torus $H_\bullet$ so that $\partial H_\bullet = \Sigma$ as oriented manifolds. Define the following ``$p$--coordinate" on each $H_\bullet$:
\begin{align*}
    p_\alpha &= p_2, \\
    p_\beta &= p_1, \\
    p_\gamma &= 3 - p_1 - p_2.
\end{align*}
Then each solid torus minus its core, $H_\bullet \setminus C_\bullet$, is explicitly parameterized, respecting orientations, as $(0,1] \times T^2$ by the coordinates $(p_\bullet,q_2,q_1)$.

Note that the symplectic vector field $S_{13}$ is transverse to $H_\alpha$, $S_{12}$ is transverse to $H_\beta$ and $S_{23}$ is transverse to $H_\gamma$. Thus each induces a closed $1$--form $\eta_\bullet = \imath_{S_{ij}} \omega|_{H_\bullet}$, and any Lagrangian surface which is tangent to $S_{ij}$ must intersect the corresponding solid torus $H_\bullet$ in curves which are tangent to the kernel of $\eta_\bullet$. These $1$--forms are as follows, expressed in the coordinates $(p_\bullet,q_2,q_1)$:
\begin{align*}
    \eta_\alpha &= dq_1, \\
    \eta_\beta &= dq_2, \\
    \eta_\gamma &= dq_1 - dq_2.
\end{align*}
The kernels of these closed $1$--forms integrate to give foliations of the handlebodies by meridional disks; we label these foliations $\mathcal{F}_\alpha$, $\mathcal{F}_\beta$ and $\mathcal{F}_\gamma$.

In fact we can do more to standardize coordinates on our handlebodies: If $p_\bullet$ is a radial coordinate on the solid torus $H_\bullet$ then we can also write down meridional and longitudinal coordinates $\mu_\bullet$ and $\lambda_\bullet$ as follows:
\begin{align*}
    \mu_\alpha &= q_2, \lambda_\alpha = q_1, \\
    \mu_\beta &= q_1 - q_2, \lambda_\beta = - q_2, \\
    \mu_\gamma &= -q_1, \lambda_\gamma = q_2 - q_1.
\end{align*}
In other words, the coordinates $(p_\bullet,\mu_\bullet,\lambda_\bullet)$ explicitly parameterize $H_\bullet$ as $D^2 \times S^1$ via $(p_\bullet,\mu_\bullet) = (r^2,\theta)$ giving polar coordinates $(r,\theta)$ on $D^2$ and $\lambda_\bullet$ being the angular coordinate on the $S^1$ factor. Each foliation $\mathcal{F}_\bullet$ is then just given by parallel meridional disks $D^2 \times \{c\}$ for $c \in S^1$, and is tangent to the kernel of the closed $1$--form $d\lambda_\bullet$. For the record, note that the symplectic form $\omega$ restricts to each meridional disk as the area form $dp_\bullet \wedge d\mu_\bullet = 2 r dr d\theta$ which is {\em twice} the standard area form on $D^2$.

\begin{proof}[Proof of \autoref{T:LagrangianCap}]
Our Lagrangian cap will be constructed so as to be invariant under the symplectic vector fields $S_{ij}$ in neighborhoods of the solid tori $H_\bullet$, so we begin the construction by applying the methods of \autoref{subsec:pointstoarcs} to turn a given geometric triple grid diagram into a system of flat arcs and associated system of Legendrian lifts in each solid torus $H_\alpha$, $H_\beta$ and $H_\gamma$, as in \autoref{P:FlatAndLegendrianTangles}. First we work with the system $A_i$ of flat arcs and $\tilde{A}_i$ of Legendrian lifts in $H_\alpha$ and we begin building our Lagrangian surface $L(D)$ in pieces, with agreement on the overlaps so as to ensure smoothness.

Let $\Delta_{1/4} \subset \Delta$ be the right triangle in the $(p_1,p_2)$--plane with vertices at $(3/4,3/4)$, $(3/4,3/2)$ and $(3/2,3/4)$ as indicated in the center of \autoref{F:VariousDeltas} (the grey triangle). Let $P \subset \Sigma$ be the set of points making up our diagram $D$. Recall that $\Pi: \CP{2} \to \Delta$ is our moment map. In $\Pi^{-1}(\Delta_{1/4}) \subset \CP{2}$, note that our toric coordinates $(p_1,p_2,q_2,q_1)$ parameterize $\Pi^{-1}(\Delta_{1/4})$ as $\Delta_{1/4} \times \Sigma$. With respect to this parameterization, let $L(D) \cap \Pi^{-1}(\Delta_{1/4}) = \Delta_{1/4} \times P$; this is just one ``flat'' copy of $\Delta_{1/4}$ for each point on the grid diagram, and is clearly Lagrangian.
\begin{figure}
    \labellist
    \small\hair 2pt
    \pinlabel {$\Delta_\alpha$} [tl] at 145 0
    \pinlabel {$R_{\alpha,1}$} [t] at 90 0
    \pinlabel {$R_{\alpha,3}$} [tl] at 210 0
    \pinlabel {$\Delta_\beta$} [br] at 0 145
    \pinlabel {$R_{\beta,1}$} [r] at 0 90
    \pinlabel {$R_{\beta,2}$} [br] at 0 210
    \pinlabel {$\Delta_\gamma$} [bl] at 158 158
    \pinlabel {$R_{\gamma,2}$} [b] at 90 230
    \pinlabel {$R_{\gamma,3}$} [l] at 230 90
    \pinlabel {$\Delta_{1/4}$} [bl] at 210 140
    \pinlabel {$E_1$} [tr] at 13 13
    \pinlabel {$E_2$} [b] at 7 280
    \pinlabel {$E_3$} [l] at 280 7
    \endlabellist
    \centering
    \includegraphics[width=6 cm]{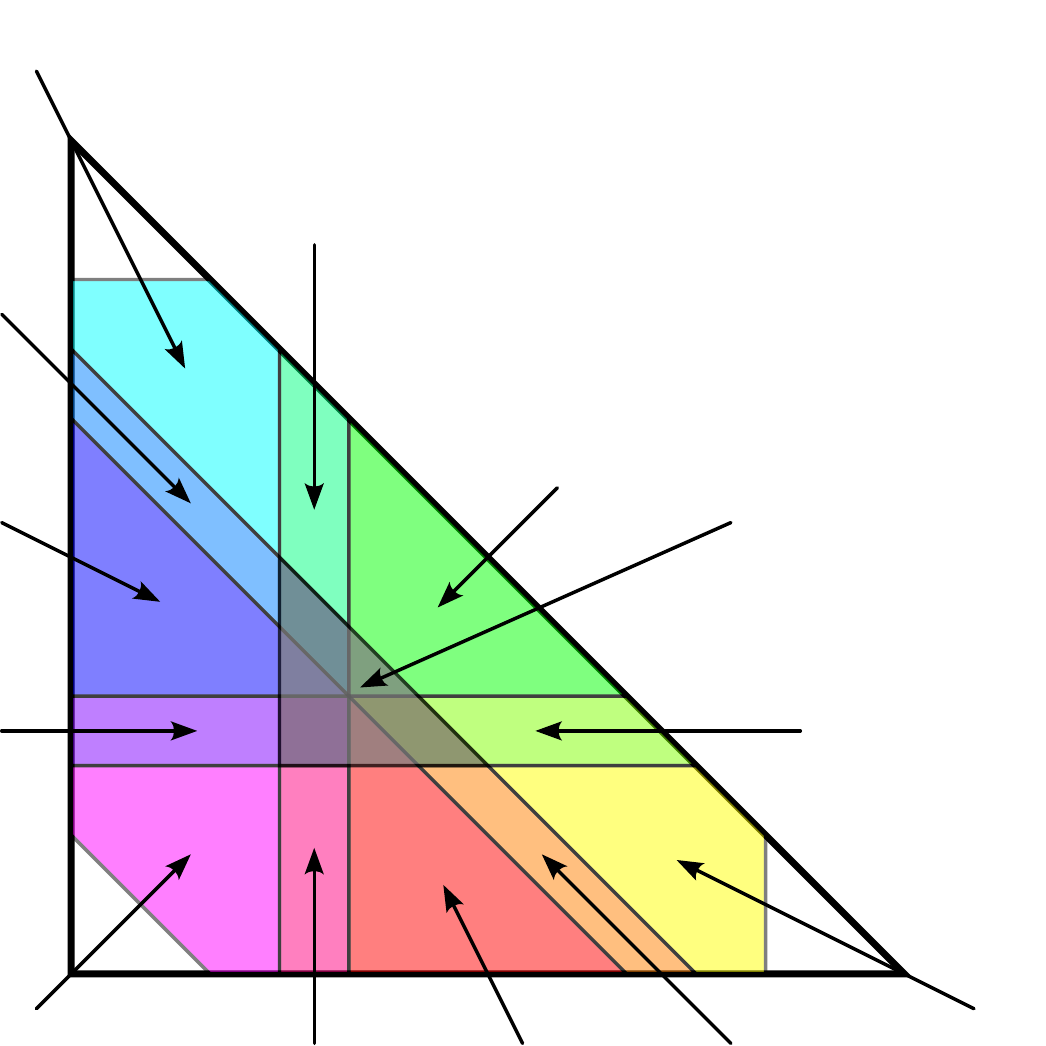}
    \caption{Various labelled parts of the moment image of $\CP{2}$, corresponding to steps in the construction of our Lagrangian triple cap.} 
    \label{F:VariousDeltas}
\end{figure}

Now let $\Delta_\alpha \subset \Delta$ be the right triangle with vertices at $(1,0)$, $(1,1)$ and $(2,0)$, shown in \autoref{F:VariousDeltas} as the red triangle at the bottom. Note that $H_\alpha \subset \Pi^{-1}(\Delta_\alpha)$ and that flow forwards and backwards along the symplectic vector field $S_{13}$ starting at $H_\alpha$ sweeps out all of $\Pi^{-1}(\Delta_\alpha)$. Let $F(D) \cap \Pi^{-1}(\Delta_\alpha)$ be the surface obtained by flowing all the chosen flat arcs $A_i$ in $H_\alpha$ forwards and backwards along this vector field $S_{13} = \partial_{p_1}$. As discussed above, this is Lagrangian because the arcs are tangent to the kernel of the $1$--form $\imath_{S_{13}} \omega |_{H_\alpha}$. Also note that, because each arc is radial on $\{p_\alpha \in [3/4,1]\}$, this definition of $L(D) \cap \Pi^{-1}(\Delta_\alpha)$ agrees with our previous definition of $L(D) \cap \Pi^{-1}(\Delta_{1/4})$ on their overlap in $\Pi^{-1}(\Delta_{1/4} \cap \Delta_\alpha)$. Appropriately applying the $S_3$ symmetry of $\CP{2}$ (permuting the homogeneous coordinates $[z_1:z_2:z_3]$) tells us how to repeat this process in $\Pi^{-1}(\Delta_\beta)$ and $\Pi^{-1}(\Delta_\gamma)$, where these regions are also shown in \autoref{F:VariousDeltas}.

Next let $R_{\alpha,1}$ be the rectangle with vertices at $(3/4,0)$, $(1,0)$, $(1,1)$ and $(3/4,1)$, shown in \autoref{F:VariousDeltas}, to the left of $\Delta_\alpha$. Here we construct $L(D)$ as a union of rectangles, one for each of the chosen arcs $A_i$ in $H_\alpha$. Recall that each such arc is parameterized by $(p_\alpha(t), \mu_\alpha(t), \lambda_\alpha(t))$ satisfying conditions outlined above. Also recall that we constructed a Legendrian lift of each of these arcs, parameterized by $(p_\alpha(t), \mu_\alpha(t), \tilde{\lambda}_\alpha(t))$. Choose a function $f: [3/4,1] \to [3/4,1]$ satisfying the following properties, with respect to some suitably small $\epsilon>0$:
\begin{itemize}
    \item for $s \in [3/4,3/4+\epsilon]$, $f(s) = s$,
    \item for $s \in [1-\epsilon,1]$, $f(s) = 1$,
    \item and for all $s\in [3/4,1]$, $f'(s) \in [0,1+\epsilon]$.
\end{itemize}
Using these, we construct a Lagrangian rectangle parameterized as follows:
\begin{align*}
    p_1(s,t) &= s, \\
    q_1(s,t) &= f'(s) \tilde{\lambda}_\alpha(t), \\
    p_2(s,t) &= f(s) p_\alpha(t), \\
    q_2(s,t) &= \mu_\alpha(t).
\end{align*}
Note that in order for this to be properly embedded in $\Pi^{-1}(R_{\alpha,1})$, the domain of this parameterization needs to be $\{(s,t) \mid s \in [3/4,1], t \in [-1/f(s),1/f(s)] \}$, which is why we extended our original parameterized arcs to allow $t \in \RR$. The fact that this rectangle is Lagrangian follows from the fact that $\tilde{\lambda}'_\alpha(t) = - p_\alpha(t) \mu'_\alpha(t)$. This is our construction of $L(D) \cap \Pi^{-1}(R_{\alpha,1})$.

We now make several observations about this construction that can be verified by direct computation.
\begin{enumerate}
    \item This construction of $L(D) \cap \Pi^{-1}(R_{\alpha,1})$ agrees with our construction of $L(D) \cap \Pi^{-1}(\Delta_{1/4})$ over the overlap $\Delta_{1/4} \cap R_{\alpha,1}$.
    \item For $s \in [1-\epsilon,\epsilon]$, where $f(s) = 1$ and $f'(s) = 0$, $L(D)$ is the same as the surface we would get by flowing our given flat arcs $A_i$ in $H_\alpha$ backwards along the symplectic vector field $\partial_{p_1}$ and thus joins smoothly with the construction of $L(D) \cap \Pi^{-1}(\Delta_\alpha)$.
    \item For $s \in [3/4,3/4+\epsilon]$, where $f(s) = s$ and $f'(s) = 1$, $L(D)$ is tangent to the Liouville vector field $V_1 = p_1 \partial_{p_1} + p_2 \partial_{p_2}$.
    \item Consider the solid torus $H_{\alpha,1} = \{(p_1,q_1,p_2,q_2) \mid p_1 = 3/4, 0 \leq p_2 \leq 3/4\}$. Use solid torus coordinates $(p_{\alpha,1},\mu_{\alpha,1}, \lambda_{\alpha,1})$ on $H_{\alpha,1}$ defined as follows: $p_{\alpha,1} = (4/3) p_2$, $\mu_{\alpha,1} = q_2$, $\lambda_{\alpha,1} = q_1$. With respect to these coordinates, $L(D) \cap H_{\alpha,1}$ is the collection of arcs parameterized by $p_{\alpha,1}(t) = \tilde{\lambda}_\alpha(t)$, $\mu_{\alpha,1}(t) = \mu_\alpha(t)$, $\lambda_{\alpha,1}(t) = \tilde{\lambda}_\alpha(t)$.
    \item In particular, each component of $L(D) \cap H_{\alpha,1}$ is an exact copy of one of the corresponding Legendrian arcs $\tilde{A}_i$ in the solid torus $H_\alpha$, and is Legendrian with respect to the contact structure induced by the Liouville vector field $V_1$.
\end{enumerate}

Now repeat these constructions over the remaining triangles $\Delta_\beta$ and $\Delta_\gamma$ and ``rectangles'' $R_{\bullet,i}$. This gives our Lagrangian $L(D)$ over all of the colored regions in \autoref{F:VariousDeltas} except the three ``ends'' $E_1$, $E_2$ and $E_3$. Since $L(D)$ is tangent to the appropriate Liouville vector fields over the outer edges of the rectangles, we extend $L(D)$ over each $E_i$ by flowing inward along the vector field $V_i$ and this completes our construction of $L(D)$.

Now we verify that the Legendrian links at the three boundary $S^3$'s are in fact Legendrian isotopic to the expected Legendrians $\Lambda_{\alpha \beta}$, $\Lambda_{\beta \gamma}$ and $\Lambda_{\gamma \alpha}$. This is clear because each of these Legendrian links is obtained by gluing together two of the given Legendrian tangles, and we have seen in \autoref{P:JoiningLegendrianTangles} that these are Legendrian isotopic to the standard Legendrianizations of our three grid diagrams.

To complete the proof, having constructed our Lagrangian surface $L(D)$, we need to show that $L(D)$ is diffeomorphic to the abstract ribbon surface $R(D)$. Referring back to~\autoref{F:VariousDeltas}, note that $L(D)$ is diffeomorphic to that part of $L(D)$ which sits over $\Delta_{1/4} \cup \Delta_\alpha \cup \Delta_\beta \cup \Delta_\gamma$, because the rest of $L(D)$ is just built by extending collars to the boundary of this much of $L(D)$. The part sitting over $\Delta_{1/4}$ is a disjoint union of oriented disks, one for each vertex of the abstract colored trivalent graph $\Gamma(D)$, as in the construction of $R(D)$ from the introduction. The orientations of these disks are inherited from the orientation of $\Delta_{1/4}$. Sitting over $\Delta_\alpha$ we have one band for each red edge of $\Gamma(D)$, over $\Delta_\beta$ one band for each blue edge, and over $\Delta_\gamma$ one band for each green edge. These bands are clearly attached in red-blue-green order as one goes clockwise around the boundary of each of the disks, because the regions $\Delta_\alpha$, $\Delta_\beta$ and $\Delta_\gamma$ are attached to $\Delta_{1/4}$ in that order. 

The last thing to verify is that each band is a locally orientation-{\em reversing} band, relative to the orientations of the disks. This is because each band can be parameterized as $[0,1] \times [-1,1]$, where the core $[0,1] \times \{0\}$ maps under the moment map to the corresponding red, blue or green straight line segment in \autoref{F:MomentMapTrisection}, with both endpoints at the center of the moment map image, and where each $\{t\} \times [-1,1]$ factor maps homeomorphically onto its image, which is a straight line segment transverse to the image of the cores $[0,1] \times \{0\}$. Thus each band connects two stacked disks, one on top of the other, {\em without} a twist, leading to a reversal of orientation.
\end{proof}

\section{Examples and applications} \label{sec:examples}

In this section we present several examples and applications. The examples mostly come in the form of combinatorial triple grid diagrams, although as noted in \autoref{sec:introduction}, a geometric triple grid diagram may be obtained immediately from a combinatorial one. 

\subsection{Orientability and Euler characteristic}

Orientability and the Euler characteristic can be determined easily from a triple grid diagram. The following definition is formulated in the language of combinatorial grid diagrams, but with slight modifications can be reformulated for geometric triple grid diagrams.

\begin{definition}
A triple grid diagram is \textit{orientable} if $X$'s and $O$'s can be placed consistently in all three diagrams. That is, $X$'s and $O$'s can be placed in each diagram such that each vertex is assigned an $X$ or an $O$, and each row, column, and diagonal that contains vertices has exactly one $X$ and one $O$. See \autoref{fig:grid}.
\end{definition}

\begin{figure}[h]
    \centering
    \includegraphics[width=1\linewidth]{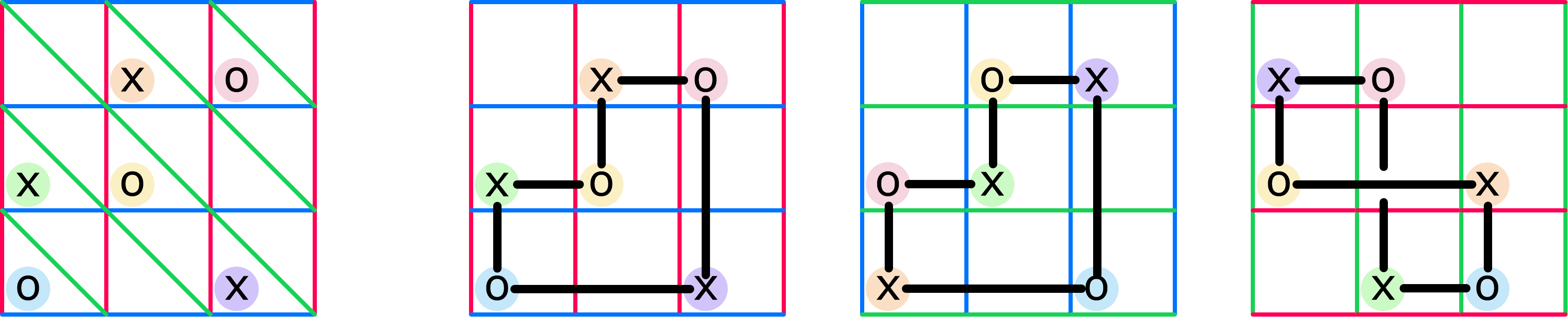}
    \caption{An orientable triple grid diagram. The right-hand side shows separately the three grids making up the combinatorial triple grid diagram on the left-hand side, along with the knot represented by each grid. Colors are used to indicate where each vertex in the triple grid diagram shows up in the three individual grids.
    \label{fig:grid}}
\end{figure}

If a triple grid diagram is orientable then it determines an orientable surface, as the compatible orientations on the three knots induce an orientation on the surface. If $X$'s and $O$'s cannot be placed consistently, then the diagram determines a nonorientable surface. 

Now we compute the Euler characteristic from a (combinatorial) triple grid diagram. Here we assume the surface is connected, but this calculation can be easily modified for multiple components. We will also carry out our computation for a closed surface, which we can think of as the surface obtained by abstractly filling each boundary component of the triple cap with a disk. Let $b$ be half the number of vertices (points) of a triple grid diagram. If there are no empty rows, columns, or diagonals, then $b$ is the same as the grid number. The Euler characteristic of the (connected) surface created from the triple grid diagram is \[\chi=V-E+F=2b-3b+F=F-b,\] where $V$ is the number of vertices, $E$ is the number of edges, and $F$ is the number of faces (which is the same as the number of link components). In particular this means that for orientable surfaces $\#^gT^2$ we have \[2-2g=F-b,\] and for nonorientable surfacs $\#^k\mathbb{RP}^2$ we have \[2-k=F-b.\]
If we do not fill in all the boundary components with disks, the above two formulae are still correct but the Euler characteristic formula needs to be adjusted appropriately.

Note that following the work in \cite{MZ1,MZ2,HKM}, if each of the three links in a triple grid diagram is an unlink, then we can (uniquely) construct a smoothly embedded surface by filling in these three unlinks with disks. In \cite{Blackwell}, such triple grid diagrams were called \textit{simple}. In general, these surfaces will not be Lagrangian; the additional $tb=-1$ condition that the diagram must satisfy in order to produce a closed Lagrangian surface is fairly restrictive. 

There exist simple triple grid diagrams representing (smoothly) embedded $\#^g \mathbb{T}^2$ and $\#^k \RP{2}$ for all integers $g,k>0$; see \autoref{ex:staircase}, \autoref{ex:klein_bottle}, and \autoref{ex:families}. The only embedded, orientable Lagrangian surface in $\CP{2}$ is a torus, as Lagrangians have isomorphic normal and tangent bundles. As for the nonorientable case, Shevchishin \cite{Shevchishin} and Nemirovski \cite{Nemirovski} showed that there is no Lagrangian embedding of the Klein bottle in $\CP{2}$. Combining work of Givental \cite{Givental}, Audin \cite{Audin}, and Dai, Ho, and Li \cite{DHL}, it follows that $\#^k \RP{2}$ admits a Lagrangian embedding in $\CP{2}$ if and only if $k \equiv 2 \pmod 4$ for $k \neq 2$, or $k \equiv 1 \pmod 4$. We can realize a Lagrangian $\RP{2}$ and Lagrangian torus with small diagrams; see \autoref{ex:2x2} and \autoref{ex:3x3}. However it remains to be seen whether all Lagrangian $\#^k \RP{2}$'s can be realized.

\begin{question}
Can all $\#^k \RP{2}$ for which there exists a Lagrangian embedding in $\CP{2}$ be realized as triple grid diagrams where each link is a unlink with $tb=-1$ components?
\end{question}

Even though every (smoothly) embedded $\#^k \RP{2}$ can be represented by a simple triple grid diagram, higher $k$ values require a higher number of vertices, and it is in general very difficult to construct diagrams with many vertices without introducing unwanted cusps, which lower $tb$ and hence makes it harder to achieve $tb=-1$.

\subsection{Examples of triple grid diagrams} \label{subsec:examples}

We now present numerous examples of (combinatorial) triple grid diagrams with various interesting properties. Each figure shows a (combinatorial) triple grid diagram, and to the right shows the three links represented by the three grids, along with the corresponding Legendrian fronts.

\begin{example}[Grid number $2$] \label{ex:2x2} There is one unique triple grid diagram (\autoref{fig:2x2}) with grid number $2$ and it represents an $\mathbb{RP}^2$. Each knot is an unknot, and each unknot, when viewed as the front projection of a Legendrian knot, has $tb=-1$. So the $\mathbb{RP}^2$ represented by this grid is Lagrangian. 
\begin{figure}[H]
    \centering
    \includegraphics[width=4 cm]{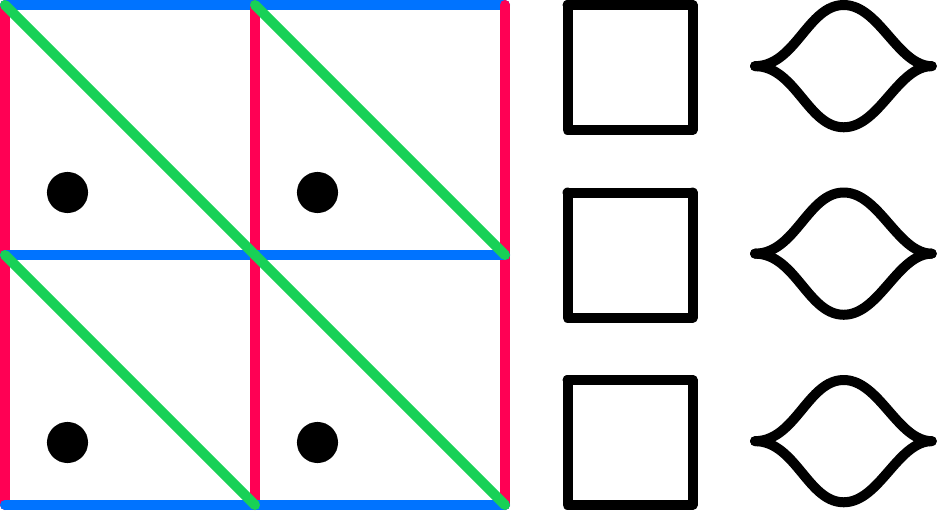}
    \caption{The unique triple grid diagram with grid number $2$.
    \label{fig:2x2}}
\end{figure}
\end{example}

\begin{example}[Grid number $3$] \label{ex:3x3} There is one unique triple grid diagram (\autoref{fig:3x3}) with grid number $3$ and it represents a $T^2$. Each knot is an unknot, and each unknot, when viewed as the front projection of a Legendrian knot, has $tb=-1$. So the $T^2$ represented by this grid is Lagrangian.
\begin{figure}[H]
    \centering
    \includegraphics[width=6 cm]{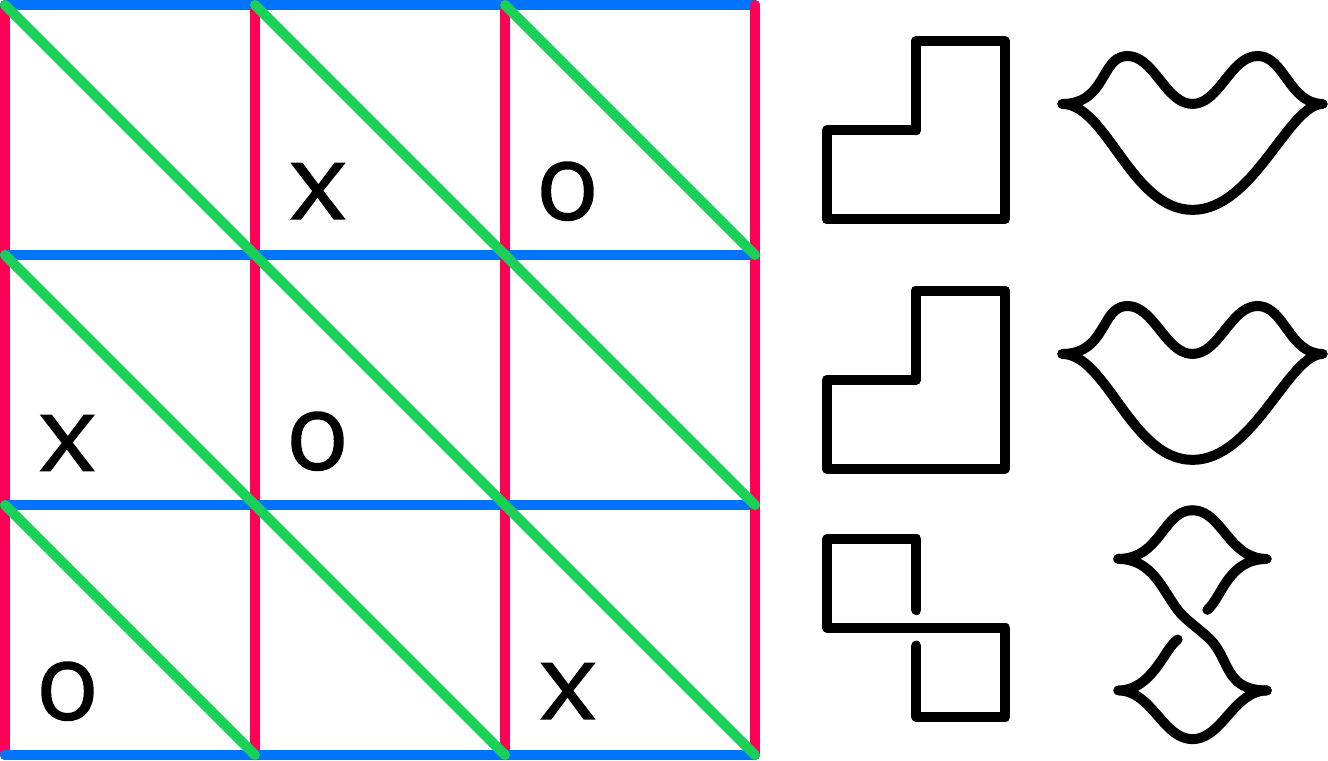}
    \caption{The unique triple grid diagram with grid number $3$.
    \label{fig:3x3}}
\end{figure}
\end{example}

\begin{example}[Another Lagrangian torus] \label{ex:torus} The following diagram (\autoref{fig:torus}) contains a two-component unlink in one grid and unknots in the others. When we cap off with disks, we obtain a $T^2$. Each of the four unknots have $tb=-1$, so this $T^2$ is Lagrangian. 
\begin{figure}[H]
    \centering
    \includegraphics[width=10 cm]{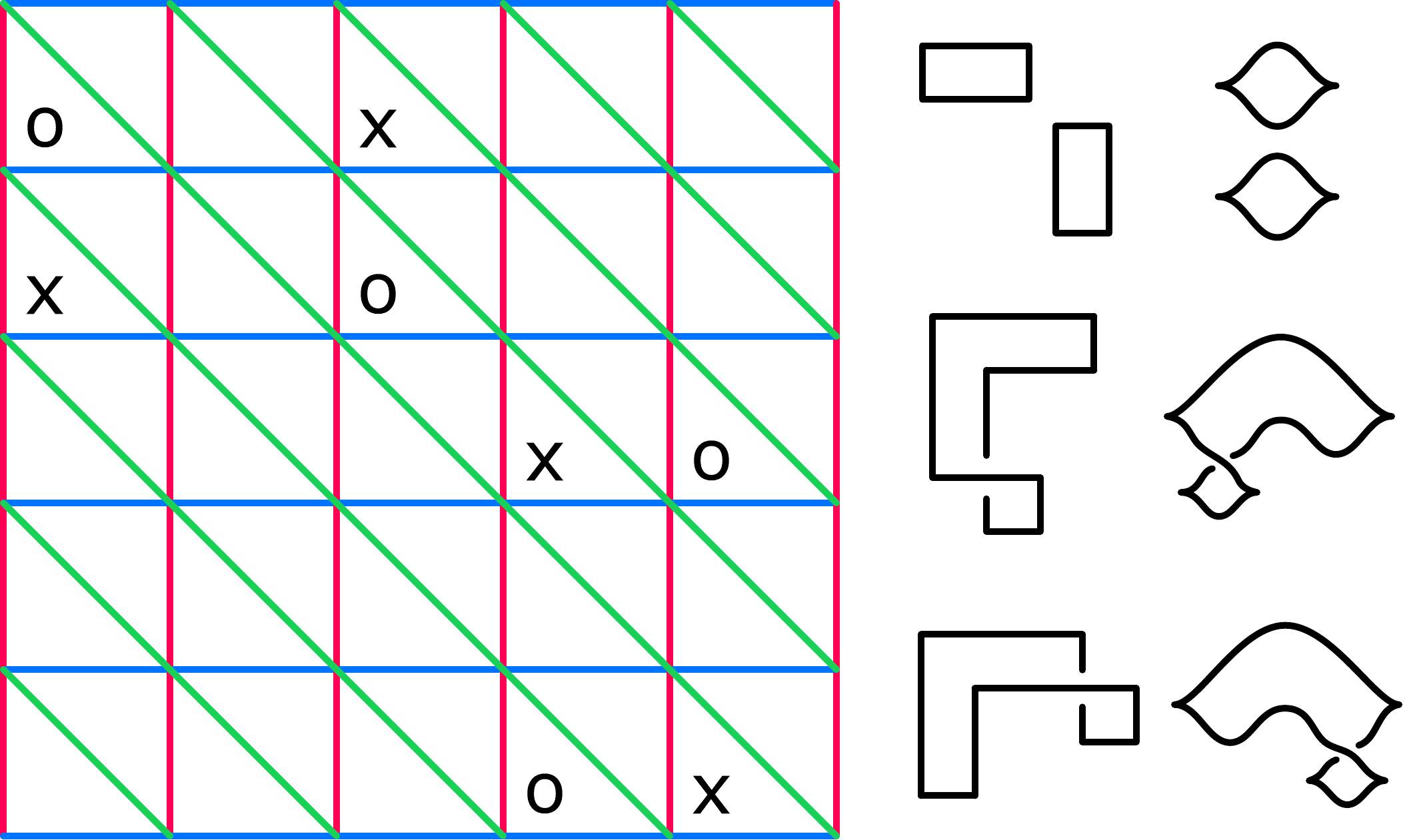}
    \caption{A triple grid diagram representing another Lagrangian torus.
    \label{fig:torus}}
\end{figure}
\end{example}

\begin{example}[Staircase family] \label{ex:staircase} Let $n$ be the grid number. A ``staircase'' triple grid diagram (\autoref{fig:staircase}), which is a diagram that looks like a staircase in one grid as shown below, produces $\mathbb{RP}^2$ for even $n$ and $\#^{\frac{n-1}{2}}T^2$ for odd $n$. The unknots making up the $\mathbb{RP}^2$ will always have $tb=-1$, but for odd grid numbers, only $n=3$ will produce $tb=-1$ unknots. That is, the even grids will always give closed Lagrangian surfaces, but the odd grids only give a closed Lagrangian surface for grid number $3$. This is good because $n=3$ yields a $T^2$ (as seen in \autoref{ex:3x3}), but higher $n$ will yield a higher genus surface, and there are no other embedded, orientable Lagrangians in $\mathbb{CP}^2$ besides $T^2$.
\begin{figure}[H]
    \centering
    \includegraphics[width=1\linewidth]{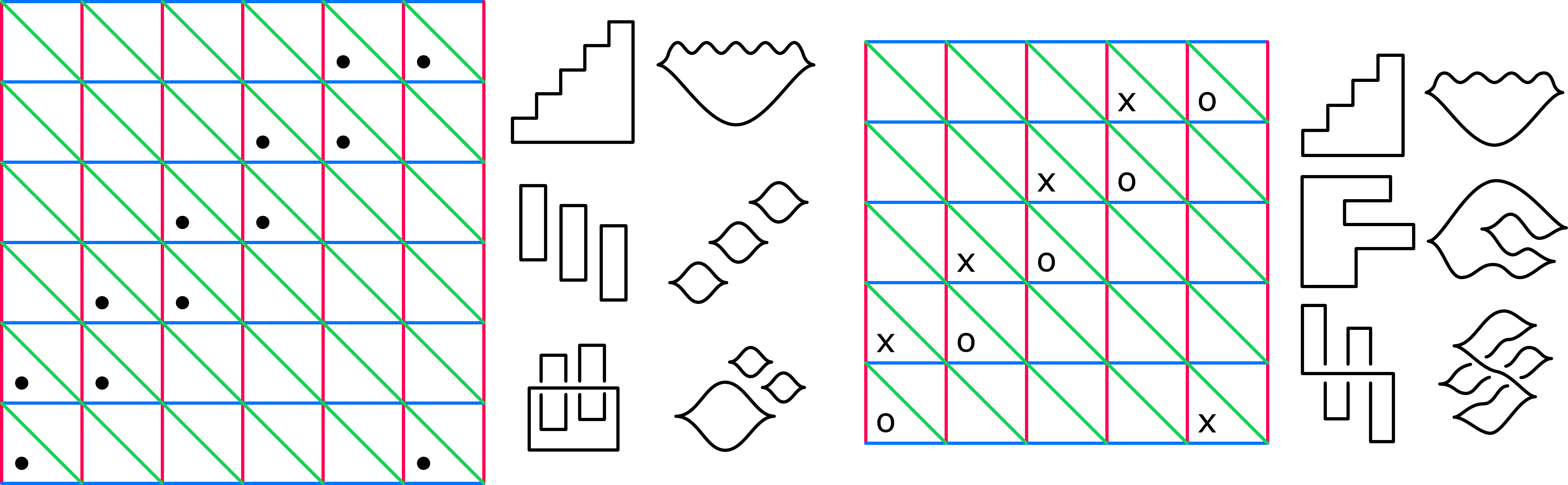}
    \caption{A family of ``staircase'' triple grid diagrams.
    \label{fig:staircase}}
\end{figure}
\end{example}

\begin{example}[Klein bottle] \label{ex:klein_bottle}
The following diagram (\autoref{fig:klein_bottle}) has an unknot in each grid, some of which do not have $tb=-1$. When we fill in with disks, we obtain a $\#^2 \mathbb{RP}^2$ which is not Lagrangian.
\begin{figure}[H]
    \centering
    \includegraphics[width=8 cm]{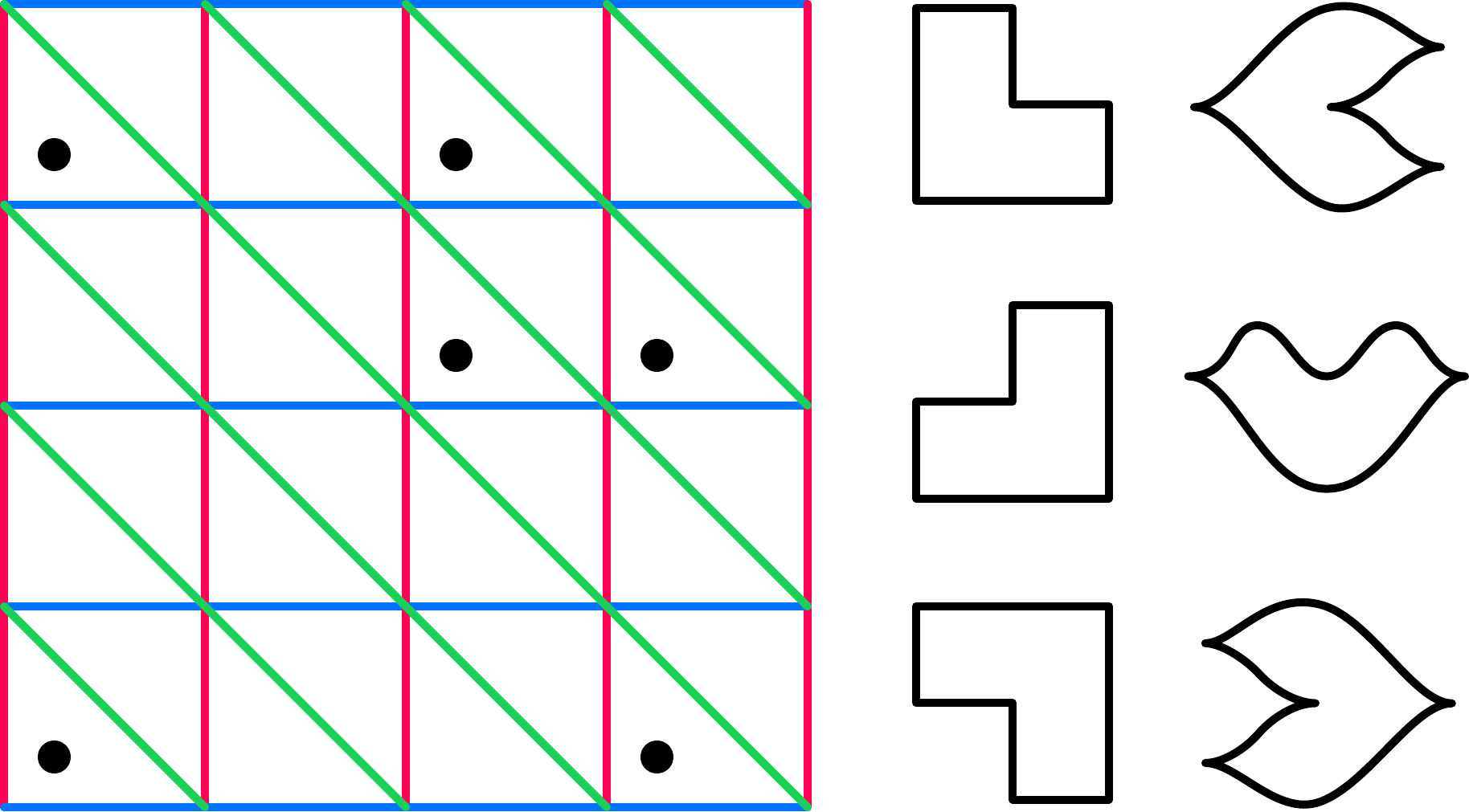}
    \caption{A triple grid diagram representing a Klein bottle.
    \label{fig:klein_bottle}}
\end{figure}
\end{example}

\begin{example}[Nonorientable families] \label{ex:families}
The following diagrams (\autoref{fig:families}) are representative of two families of diagrams which (together) produce $\#^k \mathbb{RP}^2$ for all positive integers $k$ except $2$ (see \autoref{ex:klein_bottle}). The left diagram, made up of $k$ disjoint squares along the anti-diagonal, where $k$ is odd, represents $\#^k \mathbb{RP}^2$. The right diagram, made up of $k$ disjoint squares and $k$ disjoint ``hexagons'' along the anti-diagonal, where $k$ is even, represents $\#^{k+2} \mathbb{RP}^2$. Except in the case of one square, which produces $\mathbb{RP}^2$ (see \autoref{ex:2x2}), these diagrams do not produce closed Lagrangian surfaces.
\begin{figure}[H]
    \centering
    \includegraphics[width=1\linewidth]{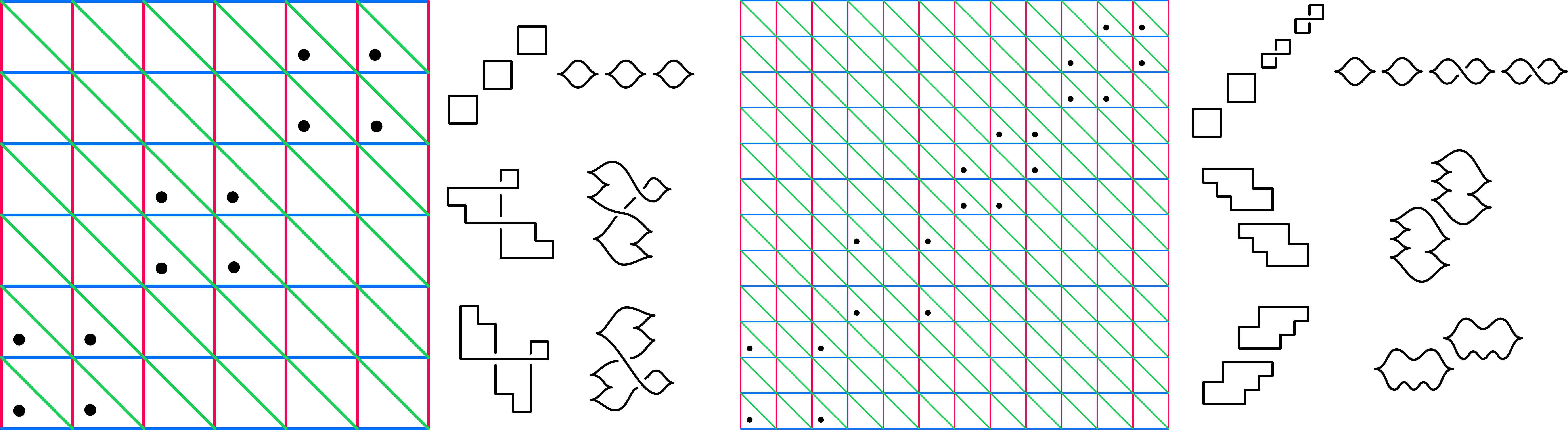}
    \caption{Two families of diagrams representing $\#^k \mathbb{RP}^2$ for all positive integers $k$ except $2$.
    \label{fig:families}}
\end{figure}
\end{example}

\begin{example}[Sphere with double point] \label{ex:sphere} The following diagram (\autoref{fig:sphere}) has a Hopf link in one grid and two-component unlinks in the others. If we cap off the Hopf link with two embedded disks with one intersection point, and the unlinks with embedded disks with no intersection points, we obtain a sphere with one double point which is Lagrangian by \autoref{C:ImmersedCase}.
\begin{figure}[H]
    \centering
    \includegraphics[width=10 cm]{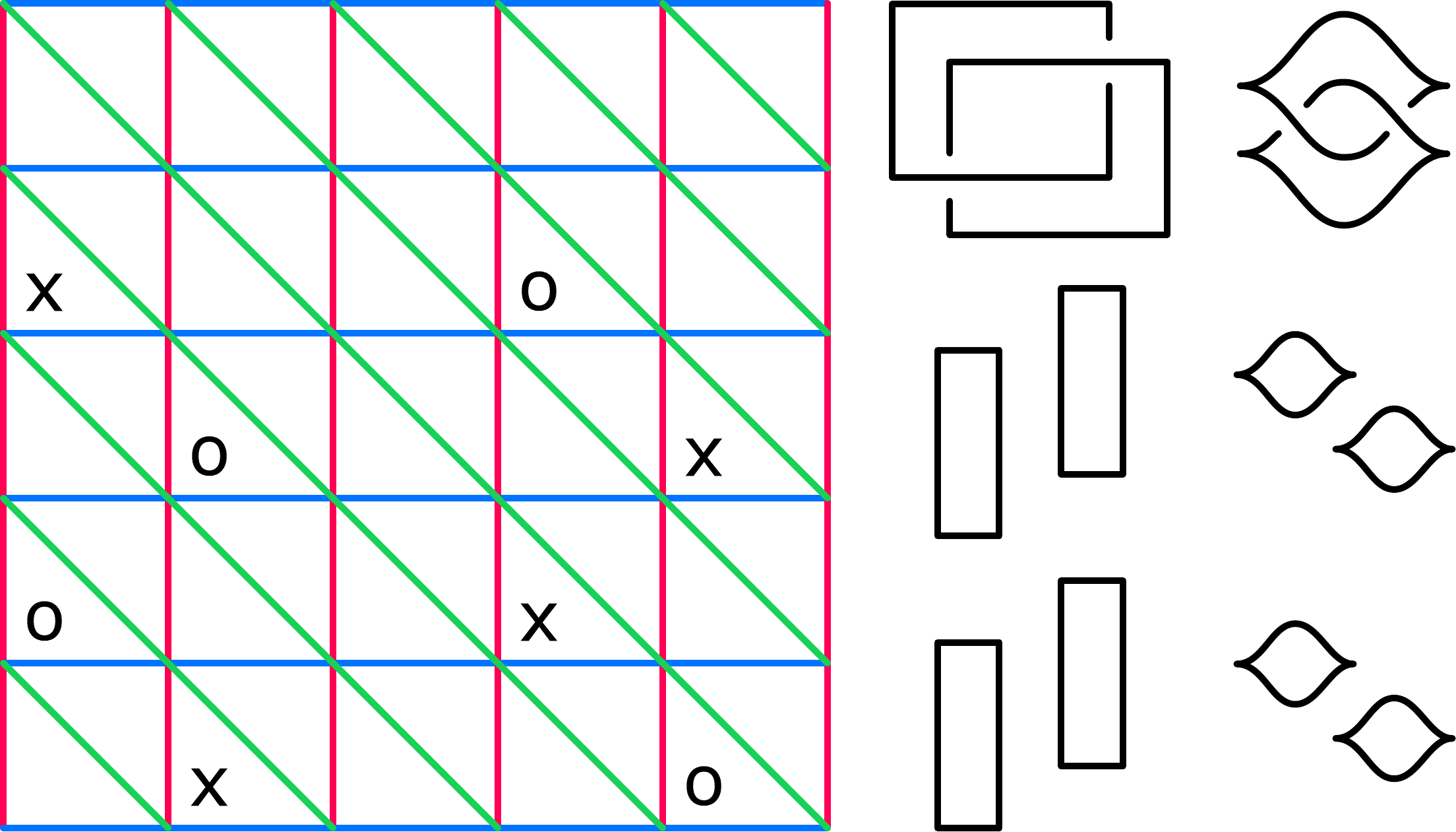}
    \caption{A triple grid diagram representing a sphere with a double point.
    \label{fig:sphere}}
\end{figure}
\end{example}

\begin{example}[Two intersecting components] \label{ex:components} The following diagram (\autoref{fig:components}) has a Hopf link in each grid. If we cap off each with two embedded disks with one intersection point, we obtain an immersed surface comprised of two embedded components which is Lagrangian by \autoref{C:ImmersedCase}. The three Hopf links are made of $tb=-1$ unknots, so the two components are both Lagrangian $\RP{2}$'s, and they intersect in three points.
\begin{figure}[H]
    \centering
    \includegraphics[width=10 cm]{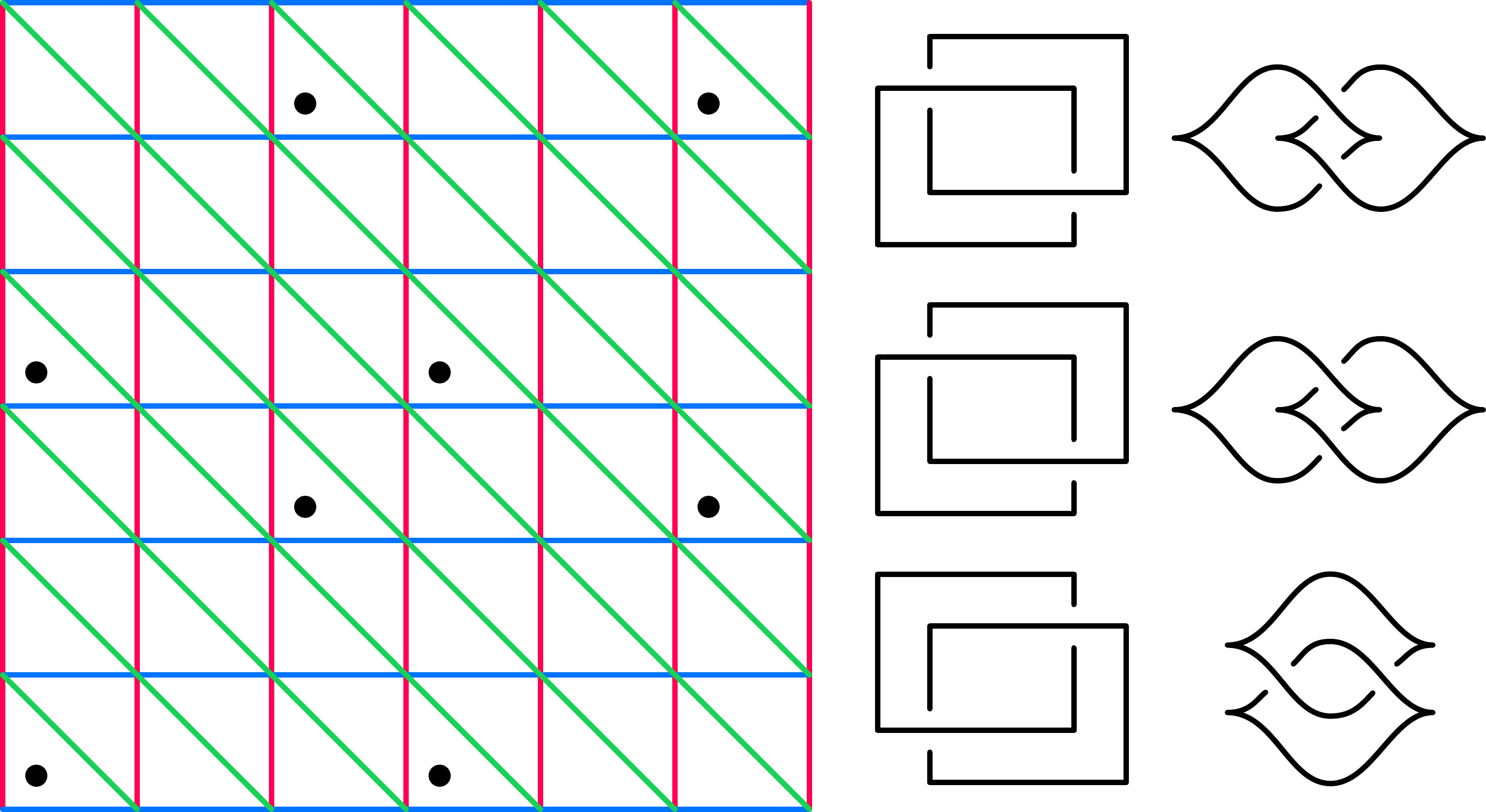}
    \caption{A triple grid diagram representing a surface with two intersecting components.
    \label{fig:components}}
\end{figure}
\end{example}

A general method for efficiently constructing {\it all} examples of triple grid diagrams has recently been introduced by Gulati and the third author \cite{GulatiLambertCole}.

\subsection{Cap for a Legendrian and its push-off}

A triple grid diagram for a cap of a Legendrian and its push-off can be constructed the following way. Start with a grid diagram for a Legendrian link, and then displace a copy of it shifted northwest by one square. Connect the vertices of the original and displaced link by diagonal lines (of slope $-1$). The result, after adding in crossings using our pairwise over/under conventions, is a triple grid diagram consisting of two copies of the original link along with a Legendrian unlink with $tb=-1$ components. See the diagram on the left in \autoref{fig:LegendrianPlusPushoff}. Filling in the $tb=-1$ unlink components with Lagrangian disks then creates a Lagrangian cap of the Legendrian and its push-off, which is topologically an annulus. 

\begin{figure}[H]
    \centering
    \includegraphics[width=10 cm]{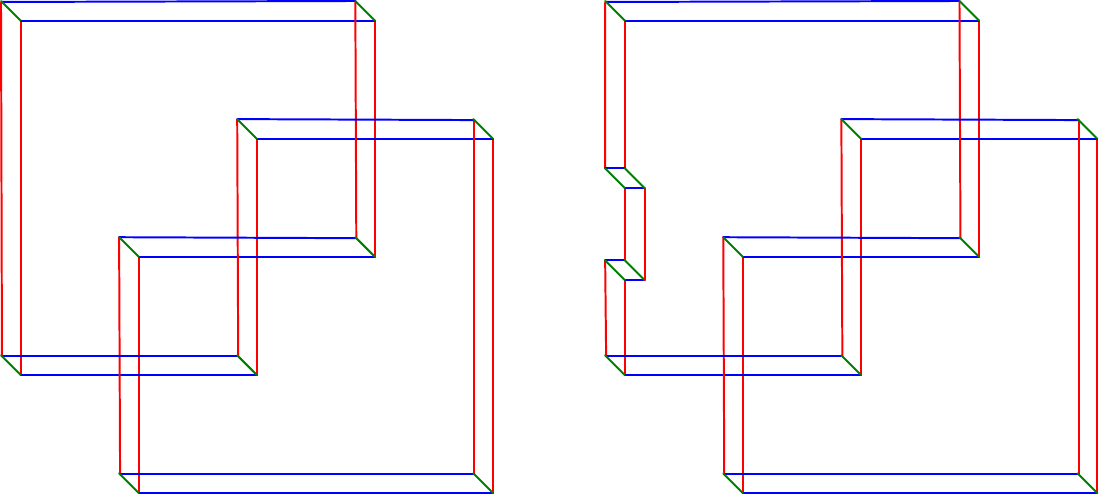}
    \caption{Turning a grid diagram for a trefoil into two different triple grid diagrams. On the left is a triple grid diagram that can be filled on the blue-green and green-red ends but which, on the red-blue end, is a copy of the original Legendrian and a Legendrian pushoff. On the right is a closely related fillable triple grid diagram. Note that here we have dropped the grid lines, so this can be seen as a geometric triple grid diagram and it is clear that there is a moduli space of such diagrams. In particular, as drawn there are sometimes multiple green edges in a given diagonal line, but these can be displaced by small perturbations. See~\cite{GulatiLambertCole} for further discussion of moduli spaces of triple grid diagrams.
    \label{fig:LegendrianPlusPushoff}}
\end{figure}

From this construction we can reprove the following result:
\begin{proposition}[Chantraine \cite{Chantraine}]
    If $\Lambda$ is a fillable Legendrian knot in $S^3$ with an orientable genus $g$ filling and Thurston-Bennequin number $\tb(\Lambda)$, then $\tb(\Lambda) = 2g - 1$.
\end{proposition}
\begin{proof}
    As in the paragraph above, construct a triple grid diagram $D$ from a grid diagram for $\Lambda$ and use this to construct an annular Lagrangrian cap for $\Lambda \cup \Lambda'$, where $\Lambda'$ is a Legendrian pushoff of $\Lambda$. Filling $\Lambda$ and $\Lambda'$ by genus $g$ Lagrangian fillings gives a closed {\em immersed} Lagrangian surface $\Sigma$ of genus $2g$ with the algebraic count of double points equal to $\tb(\Lambda)$. Since the algebraic intersection number $[\Sigma] \cdot [\Sigma]$ is $0$, the Euler number of the normal bundle $\nu \Sigma$ is equal to $-2 \tb(\Lambda)$. (See, for example, Lemma 1 in \cite{Bohr} for a statement of the standard result needed to see this.) Since $\Sigma$ is Lagrangian, its normal bundle is isomorphic to its tangent bundle, so $-2 \tb(\Lambda) = 2 - 4g$, and thus $\tb(\Lambda) = 2g-1$.
\end{proof}

The right hand diagram in \autoref{fig:LegendrianPlusPushoff} illustrates a small change that can be made in this construction to produce a fillable triple grid diagram which always gives an embedded Lagrangian torus in $\CP{2}$. Alternatively, for any Legendrian knot $\Lambda$, the link consisting of $\Lambda$ and a contact-framed push-off has an annular Lagrangian filling. Therefore, the cap and this filling can be glued to produce a closed Lagrangian torus.  

\subsection{Fillability obstructions}
The obstructions to Lagrangian embeddings of nonorientable surfaces mentioned above can in principle obstruct Lagrangian fillings of Legendrian links.  As stated, the orientability of the Lagrangian cap is determined by whether the cubic graph determined by the triple grid diagram is bipartite.  In addition, the Euler characteristic of a closed surface $L$ obtained from a triple grid diagram is determined by the formula
\[ \chi(L) = c_1 + c_2 + c_3 - b\]
where $c_1,c_2,c_3$ are the number of components of the Legendrian links $\Lambda_{\alpha \beta} (D), \Lambda_{\beta \gamma}(D)$, and $\Lambda_{\gamma \alpha}(D)$, respectively, and $2b$ is the number of points in the grid diagram.  By the classification of surfaces, this is all the information needed to determine the homeomorphism-type of the constructed surface.

\begin{theorem}
Let $D$ be a nonorientable triple grid diagram such that
\begin{enumerate}
    \item the quantity $c_1 + c_2 + c_3 - b$ is equal to 0 or is strictly negative and equal to $2$ or $3 \pmod 4$, and
    \item $\Lambda_{\alpha \beta}(D)$ and $\Lambda_{\beta \gamma}(D)$ admit Lagrangian fillings by slice disks.
\end{enumerate}
Then $\Lambda_{\gamma \alpha}(D)$ does not admit a Lagrangian filling by slice disks.
\end{theorem}

\begin{proof}
Suppose, by contradiction, that $\Lambda_{\gamma \alpha}(D)$ does admit a Lagrangian filling by slice disks.  Then by \autoref{C:ClosedLagrangian}, we obtain a closed, embedded Lagrangian surface homeomorphic to either the Klein bottle or $\#_k \RP{2}$ for $k = 0$ or $k = 3 \pmod 4$, which violates the embedding results of Givental \cite{Givental}, Audin \cite{Audin}, and Dai-Ho-Li \cite{DHL}.
\end{proof}

This result is just the tip of the iceberg for fillability obstructions, as a triple cap for three Legendrians impose various constraints on their simultaneous fillability. We leave this as an exercise to the clever reader.

\printbibliography[title={Bibliography}]

\end{document}